%% file: top-degree-Jack-characters.tex
\documentclass[12pt]{amsart}

\input{preamble.tex}

\thanks{Version identifier: \texttt{\input{git-commit-id.txt}}}

\author {Piotr \'Sniady}
\address{
Institute of Mathematics, Polish Academy of Sciences, 
\mbox{ul.~\'Sniadec\-kich 8,} \linebreak 00-956 Warszawa, Poland
} 
\email{psniady@impan.pl}

\title[Jack characters]{Asymptotics of Jack characters} 

\begin{document}

\begin{abstract}
\emph{Jack characters} are a one-parameter deformation of the characters of the symmetric groups;
a deformation given by the coefficients in the expansion of Jack symmetric functions 
in the basis of power-sum symmetric functions. 
We study Jack characters from the viewpoint of the asymptotic representation theory.
In particular, we give explicit formulas for their asymptotically top-degree part,  
in terms of bicolored oriented maps with an arbitrary face structure.
We also study their multiplicative structure and their structure constants
and we prove that they fulfill approximate factorization property, 
a convenient tool for proving Gaussianity of fluctuations of random Young diagrams.
\end{abstract}

\subjclass[2010]{%
Primary   05E05; %Symmetric functions and generalizations
Secondary 
20C30,  % Representations of finite symmetric groups
05C10,  % Planar graphs; geometric and topological aspects of graph theory for the "topological aspects"
05E10,    % Combinatorial aspects of representation theory [See also 20C30]
05E15.   % Combinatorial aspects of groups and algebras [See also 14Nxx, 22E45, 33C80]
}

\keywords{Jack polynomials, Jack characters, oriented maps, free cumulants,
Kerov polynomials, Kerov--Lassalle polynomials, structure coefficients, approximate factorization of characters}

\maketitle

For a given partition $\lambda\vdash n$
we consider the expansion of the corresponding Schur function in the basis of the power-sum symmetric functions:
\[
s_\lambda=\sum_{\pi\vdash n} \theta_{\pi}(\lambda)\ p_{\pi}.
\]
The normalized coefficient 
\[ \chi_{\lambda}(\pi):= \frac{z_\pi}{n!} \theta_{\pi}(\lambda) = \tr \rho_\lambda(\pi) \]
turns out to be equal to the irreducible character of the symmetric group,
taken with respect to the \emph{normalized trace}
\[ \tr A := \frac{\Tr A}{\Tr 1}.\]
Above,
\[ z_\pi=\prod_i i^{m_i(\pi)}\ m_i(\pi)! \]
is the standard numerical factor, where
$m_i(\pi)$ is number of the parts of $\lambda$ which are equal to $i$.
This observation is the starting point and the initial motivation 
for the following \emph{deformation
of the characters of the symmetric groups}.

\medskip

Following the ideas of Lassalle \cite{Lassalle2008a},
for a given $\alpha>0$ and a partition $\pi\vdash n$
we replace the Schur symmetric function by 
\emph{Jack polynomial} $J^{(\alpha)}_\lambda$ and consider the analogous expansion
in the basis of power-sum symmetric functions:
\begin{equation} 
\label{eq:definition-theta-A}
J^{(\alpha)}_\lambda = \sum_{\pi\vdash n} \theta^{(\alpha)}_\pi(\lambda)\ p_\pi.
\end{equation}
For partitions $\pi,\lambda\vdash n$ 
we define the \emph{irreducible Jack character $\chi^{(\alpha)}_\lambda$} as  
\begin{equation}
\label{eq:character-Jack-unnormalized-zmiana}
\chi^{(\alpha)}_\lambda(\pi) :=  \alpha^{-\frac{|\pi|-\ell(\pi)}{2}}\ \frac{z_\pi}{n!}\ \theta^{(\alpha)}_\pi(\lambda),
\end{equation}
where $\ell(\pi)$ denotes the number of parts of the partition $\pi$.
Those who like the analogy between the Jack characters and the characters of the symmetric groups
may heuristically think that the Young diagram $\lambda$ determines 
some non-existent, mythical \emph{`Jack representation'}
and the partition $\pi$ determines a conjugacy class in the symmetric group $\Sym{n}$.

\medskip

\emph{A growing collection of partial results, unproved conjectures and computer exploration indicates
that such irreducible Jack characters have a rich combinatorial and algebraic structure
which still remains elusive and resembles the one of the irreducible characters of the symmetric groups.}

In the current paper we regard Jack characters from the viewpoint
of the \emph{asymptotic representation theory} which, roughly speaking,
corresponds to the scaling in which the Young diagram $\lambda$ tends 
in some sense to infinity and the conjugacy class $\pi$ remains fixed.
With this perspective in mind our results in this paper are twofold:
firstly, we will find the first-order asymptotics of
Jack characters on a fixed conjugacy class (see \cref{sec:intro1}); secondly, we
will investigate the asymptotics of the multiplicative structure of Jack characters
and their structure constants (see \cref{sec:intro2}).

\section{Introduction part 1: asymptotics of a single Jack character}
\label{sec:intro1}

\subsection{Jack polynomials}
\label{sec:jack-polynomials-motivations}
\emph{Jack polynomials} $\big( J^{(\alpha)}_\pi\big)$ \cite{Jack1970/1971} are 
a family (indexed by an integer partition $\pi$) of symmetric functions
which depend on an additional parameter $\alpha$. 
During the last forty years, 
many connections of Jack polynomials with various fields of mathematics and physics were established: 
it turned out that the combinatorial structure of Jack polynomials plays a crucial role in 
understanding Ewens random permutations model \cite{DiaconisHanlon1992}, 
generalized $\beta$-ensembles and some statistical mechanics models 
\cite{OkounkovOlshanski1997},
Selberg-type integrals \cite{Kaneko1993},
certain random partition models 
\cite{Kerov2000,BorodinOlshanski2005,Matsumoto2008,DoleegaFeray2014}, 
and some problems of the algebraic geometry \cite{Nakajima1996},
among many others.

\subsection{Asymptotic representation theory viewpoint on Jack characters}
The usual way of viewing the characters of the symmetric groups is to fix the representation $\lambda$
and to consider the character as a function of the conjugacy class $\pi$.
However, there is also another very successful viewpoint due to Kerov and Olshanski \cite{KerovOlshanski1994}, 
called \emph{dual approach}, which suggests to
do roughly the opposite. 
We will mention only one of its success stories, namely Kerov's
Central Limit Theorem and its generalizations \cite{Kerov1993gaussian,IvanovOlshanski2002,Sniady2006c}.
Lassalle \cite{Lassalle2008a,Lassalle2009} adapted this dual approach to the framework of Jack characters.

In order for the dual approach to be successful 
one has to choose the most convenient normalization constants.
We will use the normalization introduced by Dołęga and F\'eray 
\cite{DoleegaFeray2014} 
which offers some advantages over the original normalization of Lassalle. 
Thus, with the right choice of the multiplicative constant, the irreducible Jack character
$\chi_{\lambda}^{(\alpha)}(\pi)$ becomes the \emph{normalized Jack character
$\Ch_\pi(\lambda)$}, defined as follows.
\begin{definition}
\label{def:jack-character-classical}
Let $\alpha>0$ be given and let $\pi\vdash n$ be a fixed partition. 
For a partition $\lambda\vdash N$
we define the value of the corresponding \emph{normalized Jack character}
by
\begin{equation}
\label{eq:definition-Jack}
\Ch_{\pi}(\lambda) := 
\begin{cases}
\underbrace{N (N-1) \cdots (N-n+1)}_{\text{$n$ factors}} \
\chi^{(\alpha)}_\lambda (\pi,1^{N-n})
&\text{if } N \ge n ,\\
0 & \text{if }N < n.
\end{cases}
\end{equation}
Each Jack character depends on the deformation parameter $\alpha$; 
in order to keep the notation light we make this dependence implicit.
\end{definition}
In the above definition, the irreducible Jack character $\chi^{(\alpha)}_\lambda$ 
is evaluated on $(\pi,1^{N-n})$ 
which is simply the partition $\pi$ augmented by the necessary
number of parts, all equal to $1$. 
This operation becomes very natural 
if we look on the corresponding conjugacy classes
in the symmetric groups $\Sym{N}\supseteq \Sym{n}$:
this augmentation corresponds to adding the necessary number of fixpoints (=cycles of length $1$)
to a permutation from $\Sym{n}$ so that it becomes a permutation in $\Sym{N}$.
Thus, indeed, investigation of the Jack character $\Ch_\pi$ as a function on the set 
$\AllYoung$ of Young diagrams (without any restrictions on the number of boxes) corresponds to
the scaling in which the Young diagram $\lambda$ tends to infinity while the 
\emph{`conjugacy class'} $\pi$ is fixed.

\subsection{Preliminaries: the filtered algebra $\Poly$, the embeddings}

\subsubsection{The deformation parameters. Laurent polynomials}
In order to avoid dealing with the square root of the variable $\alpha$, 
we introduce an indeterminate $A$ such that
\[ A^2 = \alpha.\]
Several quantities in this paper will be viewed as elements of $\Laurent$, 
i.e., as Laurent polynomials in the variable $A$.

\subsubsection{$\alpha$-content}
The set of Young diagrams will be denoted by $\AllYoung$.
For drawing Young diagrams we use the French convention and the usual Cartesian coordinate system;
in particular, 
the box $(x,y)\in\N^2$ is the one in the intersection of 
the column with the index $x$ and
the row with the index $y$. 
We index the rows and the columns by the elements of the set
\[\N=\{1,2,\dots\}\] 
of positive integers.

\begin{definition}
For a box $\Box=(x,y)$ of a Young diagram we define its \emph{$\alpha$-content} by
\begin{equation}
\label{eq:alpha-content}
 \acontent(\Box)=\acontent(x,y):= A x - \frac{1}{A} y\in \Laurent. 
\end{equation}
\end{definition}

\subsubsection{The algebra $\Poly$ of $\alpha$-polynomial functions on the set of Young diagrams}
\label{sec:polynomial-functions}

For an integer $n\geq 2$ we consider a function $\Tfunct_n\colon\AllYoung\to\Laurent$ given by
\[ \Tfunct_n (\lambda):= (n-1) \sum_{\Box\in\lambda} \big( \acontent(\Box) \big)^{n-2}.\]

We denote by $\Poly$ the filtered unital algebra (over the field $\Q$
of rational numbers)  which is generated by $\gamma, \Tfunct_2, \Tfunct_3, \dots$.
Above we view $\gamma$ as a constant function on $\AllYoung$ given by
\begin{equation} 
\label{eq:gamma}
\gamma := -A+\frac{1}{A}\in \Laurent.    
\end{equation}
The unit of this algebra is $1$ (=the function constantly equal to $1$).
The filtration on $\Poly$ is specified on the generators by
\begin{equation}
\label{eq:filtration}
\left\{
\begin{aligned}
\degg \gamma    &= 1,  \\  
\degg \Tfunct_n &= n \qquad \text{for $n\geq 2$};
\end{aligned}
\right.
\end{equation}
in other words the set of elements of degree at most $d$ is spanned by
\[\left\{ \gamma^{d_1} \Tfunct_2^{d_2}  \Tfunct_3^{d_3} \cdots \quad :
d_1,d_2,\ldots\geq 0,  \sum_i i d_i \leq d \right\}. \]

The elements of this algebra $\Poly$ will be called
\emph{$\alpha$-polynomial functions on the set of Young diagrams}.

\subsubsection{Number of embeddings}
\label{sec:number-of-embeddings}

Let $G$ be a \emph{bicolored graph}, i.e., a bipartite graph together 
with the choice of the coloring of the vertices.
We denote the set of its white (respectively, black) vertices
by $\V_{\circ}$ (respectively, $\V_{\bullet}$).
We will always assume that $G$ has no isolated vertices.
Furthermore, let $\lambda$ be a Young diagram.

\begin{definition}[\cite{FeraySniady2011a,DolegaFeraySniady2008}]
\label{def:embeddings}
We say that $f=(f_1,f_2)$ is an \emph{embedding} of $G$ into $\lambda$ if the functions
\[ f_1\colon \V_{\circ}\to\N, \qquad f_2\colon \V_{\bullet}\to\N\]
are such that the condition 
\begin{equation}
\label{embedding:young}
\text{$\big( f_1(w), f_2(b) \big)$ is one of the boxes of $\lambda$} 
\end{equation}
holds true for each pair of vertices $w\in \V_{\circ}$, $b\in \V_{\bullet}$ connected by an edge.
We denote by $N_G(\lambda)$ the number of embeddings of $G$ into $\lambda$.
\end{definition}

\begin{definition}
We define the \emph{normalized number of embeddings} as
\begin{equation}
\label{eq:normalized-embedding}
 \Embed_G (\lambda):=
{A}^{|\V_\circ(G)|} \left(- A^{-1}\right)^{|\V_\bullet(G)|} 
\ N_{G}(\lambda) \in\Laurent.
\end{equation}
\end{definition}

\begin{definition}
\label{def:bicolored-graph-to-permutations}
To a pair $(\sigma_1,\sigma_2)\in\Sym{n}\times \Sym{n}$ of permutations one can associate a natural bicolored graph
$G(\sigma_1,\sigma_2)$ 
with the white vertices $\V_{\circ}:=C(\sigma_1)$ corresponding to the cycles of $\sigma_1$ and  
the black vertices $\V_{\bullet}:=C(\sigma_2)$ corresponding to the cycles of $\sigma_2$.
A pair of vertices $w\in C(\sigma_1)$, $b\in C(\sigma_2)$ is connected by an edge 
if the corresponding cycles are not disjoint.

We will write 
\begin{align*}
N_{\sigma_1,\sigma_2}(\lambda) & :=N_{G(\sigma_1,\sigma_2)}(\lambda), \\  
\Embed_{\sigma_1,\sigma_2}(\lambda) & :=\Embed_{G(\sigma_1,\sigma_2)}(\lambda).  
\end{align*}

\end{definition}

\subsection{The first main result}

\subsubsection{Top-degree asymptotics of Jack characters}

We say that \emph{$\langle \sigma_1,\sigma_2 \rangle$ is transitive} 
if the group generated by the permutations $\sigma_1,\sigma_2\in \Sym{n}$
acts transitively on the underlying set $[n]=\{1,\dots,n\}$.
We define a function $\Chtt_n\colon\AllYoung\to\Laurent$ given by
\begin{equation}
\label{eq:top-top-top}
 \Chtt_n :=
\frac{-1}{(n-1)!}  \sum_{\substack{\sigma_1,\sigma_2\in \Sym{n} \\ 
\langle \sigma_1,\sigma_2 \rangle \text{ is transitive}}}  
\gamma^{n+1-|C(\sigma_1)|-|C(\sigma_2)|}
\ \Embed_{\sigma_1,\sigma_2},
\end{equation}
where $C(\pi)$ denotes the set of cycles of a permutation $\pi$.
Note that the transitivity implies that the exponent
\[ n+1-|C(\sigma_1)|-|C(\sigma_2)|\geq 0\] 
is always non-negative.

\smallskip

We will show later (in \cref{thm:degree-of-Jack-character})
that the Jack character $\Ch_n\in\Poly$ 
is of degree at most $n+1$.
The following result identifies $\Chtt_n$ defined by \eqref{eq:top-top-top}
as the top-degree part of $\Ch_n\in\Poly$.
\begin{theorem}[The first main result]
\label{theo:second-main-bis}
For each $n\geq 1$ the function
\begin{equation}
\label{eq:top-degree-of-character}
\Ch_n - \Chtt_n
\end{equation}
is an element of $\Poly$ 
of degree at most $n-1$. 
\end{theorem}
We can write this result as the following approximate equality in $\Poly$
which gives the dominant contribution for Jack characters with respect to the filtration which we consider:
\[ \Ch_n \approx \Chtt_n.\]
The proof is postponed to \cref{sec:proof}.

\subsubsection{Top-degree of Jack characters in terms of labeled maps}
\label{sec:labeled-maps}

Recall that a \emph{map} \cite{LandoZvonkin2004} 
is a graph $G$ (possibly, with multiple edges) drawn on a surface $\Sigma$.
We denote the vertex set by $\V$ and the edge set by $\E$.
As usual, we assume that $\Sigma\setminus \E$ is homeomorphic to a collection of open discs.

The sum in \eqref{eq:top-degree-of-character} is taken over the set 
\begin{equation}
\label{eq:magicset}
 \MagicSet:= \big\{ (\sigma_1,\sigma_2) \in \Sym{n} \times \Sym{n} 
: \langle \sigma_1,\sigma_2 \rangle \text{ is transitive} \big\}.   
\end{equation}
To any pair $(\sigma_1,\sigma_2)\in\MagicSet$ in this set we can canonically associate 
a map $M$ which is:
\begin{itemize}
   \item \emph{labeled, with $n$ edges}, i.e., each edge carries some label from the set 
         $[n]$ and each label is used exactly once;
   \item \emph{bicolored}, i.e., the set of vertices $\V=\V(M)$ is decomposed $\V=\V_{\circ}\sqcup \V_{\bullet}$
         into the set $\V_{\circ}=\V_{\circ}(M)$ of white vertices and the set $\V_{\bullet}=\V_{\bullet}(M)$ 
         of black vertices;
         each edge connects two vertices with the opposite colors;
   \item \emph{connected}, i.e., the graph $G$ is connected;
   \item \emph{oriented}, i.e., the surface $\Sigma$ is orientable and has some fixed orientation.
\end{itemize}
This correspondence follows from the observation that
the structure of such a map is uniquely determined by the counterclockwise cyclic order of the edges 
around the white vertices 
(which we declare to be encoded by the disjoint cycle decomposition of the permutation $\sigma_1$)
and by the counterclockwise cyclic order of the edges around the black vertices
(which we declare to be encoded by the disjoint cycle decomposition of the permutation $\sigma_2$).

\begin{example}
\label{example:map-on-torus}
The map shown in \cref{fig:torus} corresponds to the pair 
\[ \sigma_1=(1,4,9,5,7)(2,6)(3,8), \qquad \sigma_2=(1,9)(2,3,5)(4,7)(6,8).\]
\end{example}

Due to this correspondence the sum in \eqref{eq:top-top-top} can be viewed
as a summation over \emph{labeled, oriented, connected maps}.

\medskip

\begin{figure}

\subfloat[]{   
\begin{tikzpicture}[scale=0.5,
white/.style={circle,thick,draw=black,fill=white,inner sep=4pt},
black/.style={circle,draw=black,fill=black,inner sep=4pt},
connection/.style={draw=black,thick,black,auto}
]
\scriptsize

\begin{scope}
\clip (0,0) rectangle (10,10);

\fill[pattern color=\faceAfill,pattern=north west lines] (0,0) rectangle (10,10);
\fill[\faceBfill] (5,5) rectangle (8,8);

\draw (2,7)  node (b1)     [black] {};
\draw (3,2)  node (w1)     [white] {};
\coordinate (w1prim)      at (13,2);
\coordinate (w1bis)       at (3,12);
\coordinate (b1bis)       at (2,-3);

\draw (7,3)  node (bb1)    [black] {};
\coordinate (bb1prim) at (-3,3);

\draw (5,5)  node (AA)     [black] {};
\draw (8,5)  node (BA)     [white] {};
\draw (5,8)  node (AB)     [white] {};
\draw (8,8)  node (BB)     [black] {};

\draw[connection]         (w1)      to node {$5$}  (AA);
\draw[connection]         (AA)      to node {$3$}  (AB);
\draw[connection]         (AB)      to node {$8$}  (BB);
\draw[connection]         (BB)      to node {$6$}  (BA);
\draw[connection]         (BA)      to node [swap] {$2$} (AA);

\draw[connection]         (b1)       to  node [swap,pos=0.3] {$7$} (w1);

\draw[connection]         (bb1)     to node [swap,pos=0.3] {$1$} (w1prim);
\draw[connection]         (w1)      to node [pos=0.2]      {$1$} (bb1prim);

\draw[connection]         (w1)      to node[pos=0.15] {$4$}  (b1bis);
\draw[connection]         (b1)      to node [swap]    {$4$} (w1bis);

\draw[connection]         (w1)      to node [swap]    {$9$} (bb1);

\end{scope}

\draw[very thick,decoration={
    markings,
    mark=at position 0.666  with {\arrow{>}}},
    postaction={decorate}]  
(0,0) -- (10,0);

\draw[very thick,decoration={
    markings,
    mark=at position 0.666  with {\arrow{>}}},
    postaction={decorate}]  
(0,10) -- (10,10);

\draw[very thick,decoration={
    markings,
    mark=at position 0.666  with {\arrow{>>}}},
    postaction={decorate}]  
(10,0) -- (10,10);

\draw[very thick,decoration={
    markings,
    mark=at position 0.666  with {\arrow{>>}}},
    postaction={decorate}]  
(0,0) -- (0,10);

\end{tikzpicture}
\label{fig:torus}
}
\hspace{5ex}
\subfloat[]{
\begin{tikzpicture}[scale=0.5,
white/.style={circle,thick,draw=black,fill=white,inner sep=4pt},
black/.style={circle,draw=black,fill=black,inner sep=4pt},
connection/.style={draw=black,thick,black,auto}
]
\scriptsize

\begin{scope}
\clip (0,0) rectangle (10,10);

\fill[pattern color=\faceAfill,pattern=north west lines] (0,0) rectangle (10,10);
\fill[\faceBfill] (5,5) rectangle (8,8);

\draw (2,7)  node (b1)     [black] {};
\draw (3,2)  node (w1)     [white] {};
\coordinate (w1prim)      at (13,2);
\coordinate (w1bis)       at (3,12);
\coordinate (b1bis)       at (2,-3);

\draw (7,3)  node (bb1)    [black] {};
\coordinate (bb1prim) at (-3,3);

\draw (5,5)  node (AA)     [black] {};
\draw (8,5)  node (BA)     [white] {};
\draw (5,8)  node (AB)     [white] {};
\draw (8,8)  node (BB)     [black] {};

\draw[connection]         (w1)      to  (AA);
\draw[connection]         (AA)      to  (AB);
\draw[connection]         (AB)      to  (BB);
\draw[connection]         (BB)      to  (BA);
\draw[connection]         (BA)      to  (AA);

\draw[connection]         (b1)       to  (w1);
\draw[connection]         (b1)       to  (w1bis);
\draw[connection]         (bb1)      to  (w1prim);

\draw[ultra thick,dashed]  (w1)      to  (bb1);
\draw[connection]         (w1)      to  (bb1prim);
\draw[connection]         (w1)      to  (b1bis);
\end{scope}

\draw[very thick,decoration={
    markings,
    mark=at position 0.666  with {\arrow{>}}},
    postaction={decorate}]  
(0,0) -- (10,0);

\draw[very thick,decoration={
    markings,
    mark=at position 0.666  with {\arrow{>}}},
    postaction={decorate}]  
(0,10) -- (10,10);

\draw[very thick,decoration={
    markings,
    mark=at position 0.666  with {\arrow{>>}}},
    postaction={decorate}]  
(10,0) -- (10,10);

\draw[very thick,decoration={
    markings,
    mark=at position 0.666  with {\arrow{>>}}},
    postaction={decorate}]  
(0,0) -- (0,10);

\end{tikzpicture}
\label{fig:torus-rooted}
}

\caption{\protect\subref{fig:torus}: Example of a \emph{labeled} map drawn on the torus.
The left side of the square should be glued  to the right side, 
as well as bottom to top, as indicated by the arrows.
\newline
\protect\subref{fig:torus-rooted}:
The corresponding \emph{unlabeled} map.
The root edge is marked by the dashed line.}

\end{figure}
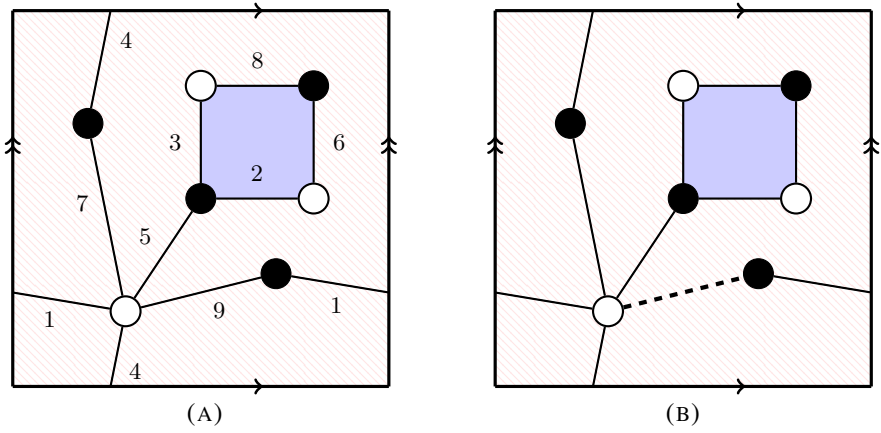

\subsubsection{Top-degree of Jack characters in terms of unlabeled maps}
Informally speaking, an \emph{unlabeled, rooted, oriented map with $n$ edges} is a labeled, oriented map, 
from which all labels have been removed,
except for a single edge. This special edge is called the \emph{root edge}. 
For an example, see \cref{fig:torus-rooted}.

This concept can be formalized as follows: on the set of labeled, oriented maps with $n$ edges 
we consider the action of the symmetric group 
\begin{equation}
\label{eq:symmetric-group-sn-1}
\Sym{n-1}:=\{\pi\in\Sym{n}: \pi(n)=n \}   
\end{equation}
by the permutation of the labels of the edges. 
An \emph{unlabeled map} is defined as an orbit of this action.
The \emph{root edge} is defined as the edge with the label $n$, which is invariant under the action of
$\Sym{n-1}$.

Such unlabeled maps are in a bijective correspondence with the equivalence classes in $\MagicSet/\sim$
with respect to the following equivalence relation:
\[ (\sigma_1,\sigma_2) \sim (\sigma_1',\sigma_2') \iff 
\bigexists_{\substack{\pi\in\Sym{n},\\ \pi(n)=n }}  \sigma'_i=\pi \sigma_i \pi^{-1} \text{ for each } i\in\{1,2\}.
\]
The equivalence classes are nothing else but the orbits of the obvious action of the group $\Sym{n-1}$
on $\MagicSet$ by coordinate-wise conjugation.

Let $\pi\in\Stab(\sigma_1,\sigma_2)\subseteq\Sym{n-1}$ 
belong to the stabilizer of some $(\sigma_1,\sigma_2)\in\MagicSet$
with respect to the above action of $\Sym{n-1}$;
in other words 
\begin{equation}
\label{eq:stabilizer}
 \sigma_i =\pi \sigma_i \pi^{-1} \text{ for each } i\in\{1,2\}.  
\end{equation}
The set of fixpoints of $\pi$ is non-empty (it contains, for example, $n$). 
Furthermore, if $x\in [n]$ is a fixpoint of $\pi$, 
then \eqref{eq:stabilizer} implies that $\sigma_i(x)$ is also a fixpoint. 
As $\langle \sigma_1,\sigma_2 \rangle$ is transitive, it follows 
that all elements of $[n]$ are fixpoints, thus $\pi=\id$. In this way we proved that 
$\Stab(\sigma_1,\sigma_2)=\{\id\}$, thus each equivalence class consists of exactly 
$\frac{\left| \Sym{n-1} \right|}{\left| \Stab(\sigma_1,\sigma_2) \right|}=(n-1)!$ elements. 
Since the number of embeddings $\Embed_{\sigma_1,\sigma_2}$ is constant on each 
equivalence class, we have proved the following result.

\begin{corollary}
\label{coro:nonoriented-maps}
The top-degree of Jack character \eqref{eq:top-top-top} can be written as
a sum over 
\emph{rooted, oriented, bicolored, connected maps $M$ with $n$ unlabeled edges}:
\begin{equation}
\label{eq:top-top-top2}
 \Chtt_n =
(-1)  \sum_{M}  
\gamma^{n+1-|\V(M)|}
\ \Embed_{M}.
\end{equation}
\end{corollary}

\subsubsection{Application: Kerov--Lassalle polynomials}

In the context of the asymptotic representation theory
a convenient way of parametrizing the shape of a Young diagram 
is provided by \emph{free cumulants} \cite{Biane1998}.
For an integer $n\geq 2$ the corresponding free cumulant
$\Rfunct_n\colon\AllYoung\to\Laurent$ is a function on the set of Young diagrams
defined as
    \begin{equation}
    \label{eq:definition-free-cumulant}    
    \Rfunct_k(\lambda) := (-1) \sum_{\sigma_1,\sigma_2}  \Embed_{\sigma_1,\sigma_2}(\lambda),
    \end{equation}
where sum in \eqref{eq:definition-free-cumulant} runs over 
pairs of permutations $\sigma_1,\sigma_2\in\Sym{k-1}$ with the property that: 
     \begin{enumerate}[label=(\alph*)]
        \item \label{item:free-cumulant-A} their product $\sigma_1 \sigma_2=(1,2,\dots,k-1)$ is the full cycle, and
        \item \label{item:free-cumulant-B} their total number of cycles fulfills $|C(\sigma_1)|+|C(\sigma_2)|=k$.
     \end{enumerate} 
Such pairs $(\sigma_1,\sigma_2)$ can be identified with plane rooted trees.

It turns out (cf.~\cref{prop:generate-the-same=free})
that the filtered unital algebra $\Poly$ can be alternatively viewed as generated by
$\gamma,\Rfunct_2,\Rfunct_3,\dots$ with the degrees of the generators 
\begin{equation}
\label{eq:filtration-5}
\left\{
\begin{aligned}
\degg \gamma    &= 1,  \\  
\degg \Rfunct_n &= n  \qquad \text{for $n\geq 2$}.
\end{aligned}
\right.
\end{equation}

\medskip

Each Jack character can be expressed in terms of these generators
by \emph{Kerov--Lassalle polynomials}; 
for example 
\begin{equation}
\label{eq:Kerov-example}
\Ch_4 =\underbrace{\Rfunct_{5} + 6 \Rfunct_{4} \gamma + \Rfunct_{2}^{2} \gamma + 11 \Rfunct_{3} \gamma^{2} + 6
\Rfunct_{2} \gamma^{3}}_{\Chtt_4} + 5 \Rfunct_{3} + 7 \Rfunct_{2} \gamma.
\end{equation}
Some partial theoretical results \cite{Lassalle2009} as well as computer explorations indicate
that all coefficients of such a polynomial for $\Ch_n$ are conjecturally non-negative integers.

\smallskip

There are some fairly standard techniques \cite{DolegaFeraySniady2008} which can be used to find an explicit form 
of Kerov--Lassalle polynomial for $\Chtt_n$ given by \cref{coro:nonoriented-maps}.
In particular, it follows that such a Kerov--Lassalle polynomial is homogeneous of degree $n+1$
and is a sum of monomials with non-negative integer coefficients.
\cref{theo:second-main-bis} implies therefore that \emph{the homogeneous part of degree $n+1$
of Kerov--Lassalle polynomial for $\Ch_n$ is equal to 
the analogous Kerov--Lassalle polynomial for $\Chtt_n$}.
In the example \eqref{eq:Kerov-example} this homogeneous part was indicated by the curly braces.

\begin{corollary}
\label{coro:Kerov-Lassalle}
For each integer $n\geq 1$, the homogeneous part of degree $n+1$ 
of Kerov--Lassalle polynomial for $\Ch_n$ 
is a sum of monomials in the generators \eqref{eq:filtration-5} 
with \emph{non-negative integer coefficients}.
\end{corollary}
The details of the proof are postponed to \cref{sec:proof-theo-kerov-lassalle}.

\subsubsection{Top-degree of Jack characters in terms of weighted unicellular maps}
We refer to \cite{Czyzewska-Jankowska2017} for an equivalent formula for
$\Chtt_n$ which is expressed in terms of \emph{unicellular non-oriented maps}, weighted 
according to some specific measure of non-orientability. 

It is also
worth to mention the long-standing open problem formulated by
Goulden and Jackson \cite{Goulden1996} known as the
$b$--conjecture. Goulden and Jackson defined certain
rational functions related to Jack polynomials and they conjectured
that, in fact, they are polynomials
with nonnegative integer coefficients. Dołęga \cite{Dolega2017a} found
the top-degree part of these expressions with striking similarities to our formulas. Although we cannot
translate one of these results into the other, we cannot resist to state
that there has to be a strong connection between both problems.

\section{Introduction part 2: the multiplicative structure of Jack characters}
\label{sec:intro2}

\subsection{Conditional cumulants}
\label{sec:conditional-cumulants}
Let $\Alg$ and $\AlgB$ be commutative unital algebras 
and let $\EE\colon\Alg\to\AlgB$ be a unital linear map.
We will refer to $\EE$ as a \emph{conditional expectation}. 

For any tuple $x_1,\dots,x_n\in\Alg$ we define their \emph{conditional cumulant} as
\begin{multline} 
\label{eq:what-is-cumulant}
\condKumu{\Alg}{\AlgB}(x_1,\dots,x_n) =  
[t_1 \cdots t_n] \log \EE e^{t_1 x_1+\dots+t_n x_n} = \\
\left. \frac{\partial^n}{\partial t_1 \cdots \partial t_n}
\log \EE e^{t_1 x_1+\dots+t_n x_n} \right|_{t_1=\cdots=t_n=0} \in\AlgB
\end{multline}
where the operations on the right-hand side should be understood 
in the sense of formal power series in the variables
$t_1,\dots,t_n$.

\subsection{Approximate factorization property}

\begin{definition}
Let $\Alg$ and $\AlgB$ be filtered unital algebras and let $\EE:\Alg\rightarrow\AlgB$ be a 
unital linear map.
We say that \emph{$\EE$ has approximate factorization property}
\cite{Sniady2006c}
if for all $l\geq 1$ and all choices of $x_1,\dots,x_l\in\Alg$ we have that
\[ \degg_{\AlgB} k_{\Alg}^{\AlgB}(x_1,\dots,x_l) \leq 
\left(\degg_{\Alg} x_1\right) + \cdots + \left(\degg_{\Alg} x_l \right)-2(l-1).\]
\end{definition}

\subsection{Disjoint product}
We introduce a parameter
\[ \delta := - \gamma, \]
cf.~\eqref{eq:gamma}, on which Jack characters depend implicitly.
We will show later in \cref{prop:filtration-by-characters}
that if we regard $\Poly$ as a $\Q[\delta]$-module, it is a free module with the basis $(\Ch_\pi)$
where $\pi$ runs over the set of partitions.
We will also show that 
with respect to this basis the filtration on $\Poly$ from \eqref{eq:filtration} 
can be equivalently defined by
\begin{equation}
\label{eq:filtration-2}
\left\{
\begin{aligned}
\degg \delta    &= 1,  \\  
\degg \Ch_{\pi} &= |\pi|+\ell(\pi) \qquad \text{for any partition $\pi$.}
\end{aligned}
\right.
\end{equation}

\medskip

This allows us to define a new multiplication on $\Poly$ (which we call \emph{disjoint product}) 
by setting on the generators
\[ \Ch_{\pi} \bullet \Ch_{\sigma} := \Ch_{\pi\sigma}, \]
where $\pi \sigma$ denotes the \emph{concatenation} of the partitions $\pi$ and $\sigma$.
For example,
\[ \Ch_{4,3,1} \bullet \Ch_{5,4,2} = \Ch_{5,4,4,3,2,1}. \]
It is easy to check that this product is commutative and associative;
the linear space of $\alpha$-polynomial functions equipped with this multiplication becomes 
an algebra which will be denoted by $\Poly_\disjoint$.
Thanks to \eqref{eq:filtration-2} is easy to check that the usual filtration 
\eqref{eq:filtration-2}
works fine also with this product; in this way $\Poly_\disjoint$ becomes a filtered algebra.

\subsection{Cumulants $\kumuDisjointPoint$}
We consider the filtered unital algebras $\Poly_\disjoint$ and
$\Poly$, and as a conditional expectation between them we take the identity map
\begin{equation}
\label{eq:commutative-diagram-A}
\begin{tikzpicture}[node distance=2cm, auto, baseline=-0.8ex]
  \node (A) {$\Poly_\disjoint$};
  \node (B) [right of=A] {$\Poly.$};

  \draw[->] (A) to node {$\id$} (B);
\end{tikzpicture}
\end{equation}
The corresponding cumulants will be denoted by $\kumuDisjointPoint$.

Computer exploration suggests that the expansions of the cumulants 
$\kumuDisjointPoint$ in the module basis \eqref{eq:filtration-2}
take a form which is interesting from the viewpoint of algebraic combinatorics
and encourages stating the following conjecture.
\begin{conjecture}
\label{conj:connection-coefficients-steroids}
For partitions $\pi_1,\dots,\pi_\ell$ we consider the expansion
\begin{equation}
\label{eq:steroids}
\kumuDisjointPoint(\Ch_{\pi_1},\dots,\Ch_{\pi_\ell}) = \sum_{\sigma} d^{\sigma}_{\pi_1,\dots,\pi_\ell}(\delta)
\ \Ch_\sigma.   
\end{equation}

Then $(-1)^{\ell-1} d^{\sigma}_{\pi_1,\dots,\pi_\ell} \in \Q[\delta]$ 
is a polynomial with non-negative integer coefficients.
\end{conjecture}

\subsection{The second main result: 
approximate factorization property}
\label{sec:main-cumulants}

\cref{conj:connection-coefficients-steroids} is beyond our reach.
Nevertheless we will prove a partial result about the form of the left-hand side of \eqref{eq:steroids},
namely that $\kumuDisjointPoint(\Ch_{\pi_1},\dots,\Ch_{\pi_\ell})\in\Poly$ is of degree at most
\[ \left( \sum_{i}  |\pi_i| + \ell(\pi_i) \right) - 2(\ell-1).  \]
This concrete claim can be reformulated in an abstract language as follows.
\begin{theorem}[The second main result]
\label{theo:factorization-of-characters}
The identity map
\begin{equation*}
\begin{tikzpicture}[node distance=2cm, auto, baseline=-0.8ex]
  \node (A) {$\Poly_\disjoint$};
  \node (B) [right of=A] {$\Poly.$};

  \draw[->] (A) to node {$\id$} (B);
\end{tikzpicture}
\end{equation*}
has approximate factorization property.

\end{theorem}
In the special case of the characters of the symmetric groups (i.e.~$A=1$, $\gamma=0$) this result was proved in \cite{Sniady2006c}.
The proof will be presented in \cref{sec:key-tool-proof}.
In a forthcoming joint paper with Dołęga
\cite{DolegaSniady2014} we will present applications of this result to
investigations of some natural models of random Young diagrams related to Jack polynomials.

\section{Heuristics: Towards the proof}

\subsection{The key tool}
\label{sec:keykeykeytool}
The main difficulty in both main results of this paper
(\cref{theo:second-main-bis} and
\cref{theo:factorization-of-characters})
is to show that a given $\alpha$-polynomial function $F\in\Poly$ is of smaller degree than
one would expect from some trivial bounds.
Our key tool will be \cref{lem:keylemma}
which provides three conditions:
\ref{zero:topdegree},
\ref{zero:vanishing}, and
\ref{zero:laurent}
which together guarantee that $F\in\Poly$ is of smaller degree than
initially suspected. The only really troublesome of them is condition \ref{zero:vanishing}
and we shall discuss it in the following.

\subsection{Finite difference operators}
\label{sec:why-finite-difference-nice}
  
Roughly speaking, the latter condition \ref{zero:vanishing}
is formulated in terms of the finite difference operators $\Delta_{\lambda_1},\Delta_{\lambda_2},\dots$
adapted to the context of functions $F(\lambda_1,\lambda_2,\dots)$ on the set $\AllYoung$ of Young diagrams.
This approach is hardly surprising as the finite difference operators have a long record of being useful
in combinatorics. In our context when $F\in\Poly$ is an $\alpha$-polynomial function,
and henceforth
$(\lambda_1,\dots,\lambda_k)\mapsto F(\lambda_1,\dots,\lambda_k)$ is a multivariate polynomial
in the lengths of the rows of a Young diagram, it is convenient that the application of
each finite difference operator decreases the degree of this multivariate polynomial by one.

\subsection{The difficulty}
The subtle issue is that 
for a function $F\colon \AllYoung\to\Laurent$ on the set of Young diagrams,
the evaluation on a Young diagram $(\lambda_1,\dots,\lambda_k)$ 
\begin{multline} 
\label{eq:finite-difference-explanation}
\left( \Delta_{\lambda_1} \cdots \Delta_{\lambda_k} F\right)(\lambda_1,\dots,\lambda_k)= \\
\sum_{\epsilon_1,\dots,\epsilon_k\in\{0,1\} } (-1)^{k-(\epsilon_1+\cdots+\epsilon_k)}  
F(\lambda_1+\epsilon_1,\dots,\lambda_k+\epsilon_k)     
\end{multline}
is a linear combination (with integer coefficients) of the values of $F$ on 
vectors $(\lambda_1+\epsilon_1,\dots,\lambda_k+\epsilon_k)$ which might \emph{not} be Young diagrams;
therefore the function $F$ might be not well-defined there.

The way to overcome this difficulty is to extend in some convenient way
the domain of the multivariate function
\[ \AllYoung\ni (\lambda_1,\dots,\lambda_k) \mapsto F(\lambda_1,\dots,\lambda_k); \]
for the extension 
\begin{equation}
\label{eq:extension-o-co-chodzi}
 \N_0^k \ni (\lambda_1,\dots,\lambda_k) \mapsto F^{\sym}(\lambda_1,\dots,\lambda_k)    
\end{equation}
the corresponding analogue of \eqref{eq:finite-difference-explanation} is well-defined.

Regretfully, this extension \eqref{eq:extension-o-co-chodzi} is no longer given by
a multivariate polynomial. For this reason it is not clear if the virtues
of the finite difference operators which we discussed in \cref{sec:why-finite-difference-nice}
are still applicable. Since the objects which we work with are no longer polynomials,
we cannot say that the finite difference operator decreases their degree.

\subsection{Solution: row functions}
A solution which we present in the current paper 
is to replace the notion of multivariate polynomials
by a larger algebra of functions 
(\emph{the algebra of row functions $\functions$}) which would be more compatible with the 
aforementioned procedure of extension of the domain.
The difficulty is to define the filtration on this algebra in such a way that
the application of the finite difference operator would still 
decrease the degree.

\subsection{Content of the paper}

The only information about the Jack characters that
is necessary for our purposes is contained in the work of
Dołęga and F\'eray \cite{DoleegaFeray2014}.
We review their findings in \cref{sec:vk-scaling}.
\cref{sec:RST} provides technical tools 
which are necessary to translate the results of Dołęga and F\'eray
to our notations.

In \cref{sec:row-functions} we introduce the aforementioned algebra $\functions$
of row functions. 
In \cref{sec:how-to-prove-small} we state the key tool
(which was discussed in \cref{sec:keykeykeytool}) for proving the degree bounds for 
$\alpha$-polynomial functions.

\cref{sec:afp-vertical,sec:cumulants-long-diagonal,sec:key-tool-proof}
are devoted specifically to the proof of the second main result,
\cref{theo:factorization-of-characters}.

\cref{sec:degree-n+1,sec:proof} are devoted specifically to the proof of the first main result, 
 \cref{theo:second-main-bis}.

\section{Preliminaries: Various functionals of Young diagrams}
\label{sec:RST}

\subsection{Smooth functionals of shape}
For an integer $n\geq 2$ we define the \emph{(anisotropic) functional of shape}
\begin{equation}
\label{eq:smooth-functional}
 \Sfunct_n (\lambda):= 
(n-1) \iint_{(x,y)\in\lambda} \big( \acontent(x,y) \big)^{n-2}\, \mathrm{d}x\, \mathrm{d}y \in\Laurent,
\end{equation}
where the integral is taken over the Young diagram $\lambda$ viewed as a subset of $\R^2$; 
in other words it is an integration over $x$ and $y$ such that
\[ y> 0 \qquad \text{and} \qquad 0< x \leq \lambda_{\lceil y \rceil}.\]

\subsection{Anisotropic Stanley polynomials}

\subsubsection{Multirectangular coordinates}
\label{sec:multirectangular-coordinates}

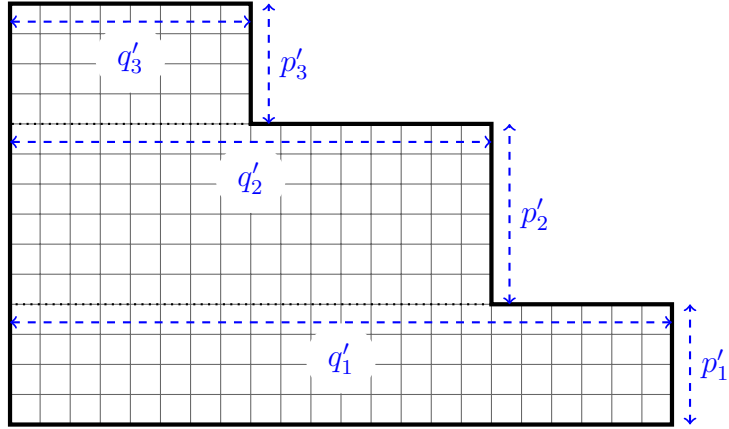
\begin{figure}[t]
\begin{tikzpicture}[scale=0.8]

\begin{scope}
\clip (0,0) -- (11,0) -- (11,2) -- (8,2) -- (8,5) -- (4,5) -- (4,7) -- (0,7) -- cycle;
\draw[black!60] (0,0) grid[step=0.5] (20,10); 
\end{scope}

\draw[ultra thick](0,0) -- (11,0) -- (11,2) -- (8,2) -- (8,5) -- (4,5) -- (4,7) -- (0,7) -- cycle;

\draw[dotted,thick] (8,2) -- (0,2)
(4,5) -- (0,5);

\begin{scope}[<->,thick,auto,dashed,blue]
\draw (11.3,0) to node[swap] {$p'_1$} (11.3,2);
\draw (0,1.7) to node[swap,shape=circle,fill=white] {$q'_1$} (11,1.7);

\draw (8.3,2) to node[swap] {$p'_2$} (8.3,5);
\draw (0,4.7) to node[swap,shape=circle,fill=white] {$q'_2$} (8,4.7);

\draw (4.3,5) to node[swap] {$p'_3$} (4.3,7);
\draw (0,6.7) to node[swap,shape=circle,fill=white] {$q'_3$} (4,6.7);
\end{scope}

\end{tikzpicture}

\caption{Multirectangular Young diagram $P'\times Q'$.}

\label{fig:multirectangular}
\end{figure}

We start with \emph{anisotropic multirectangular coordinates} $P=(p_1,\dots,p_\ell)$ and $Q=(q_1,\dots,q_\ell)$.
They give rise to \emph{isotropic multirectangular coordinates} given by
\begin{align*}
P'=(p'_1,\dots,p'_\ell):&= \left( A p_1, \dots, A p_\ell \right),\\
Q'=(q'_1,\dots,q'_\ell):&= \left( \frac{1}{A} q_1, \dots, \frac{1}{A} q_\ell \right).
\end{align*}

Suppose that $A\in\R\setminus\{0\}$ and $P,Q$ are such that 
$P'=(p'_1,\dots,p'_\ell)$ and $Q=(q'_1,\dots,q'_\ell)$ are sequences of non-negative integers 
such that $q'_1\geq \cdots \geq q'_\ell$; we consider the \emph{multirectangular} Young diagram
\[P' \times Q' = (\underbrace{q'_1,\dots,q'_1}_{\text{$p'_1$ times}},\dots,
\underbrace{q'_\ell,\dots,q'_\ell}_{\text{$p'_\ell$ times}}).\]
This concept is illustrated in \cref{fig:multirectangular}. 

\subsubsection{Anisotropic Stanley polynomials}
\label{subsec:stanley-polynomials}
Let $\St=(\St_1,\St_2,\dots)$ be a sequence of polynomials such that for each $\ell\geq 1$ 
\[\St_\ell=\St_\ell(\gamma;p_1,\dots,p_\ell;q_1,\dots,q_\ell)=\St_\ell(\gamma; P; Q) \] 
is a polynomial in $2\ell+1$ variables and
\[ \St_{\ell+1}(\gamma;p_1,\dots,p_\ell,0;q_1,\dots,q_\ell,0)=\St_\ell(\gamma;p_1,\dots,p_\ell;q_1,\dots,q_\ell). \]
We assume furthermore that the degrees of the polynomials $\St_1,\St_2,\dots$ are uniformly bounded by some integer $d$;
we say then that the degree of $\St$ is at most $d$.

\begin{definition}
\label{def:stanley-polynomial}
Let $F\colon \AllYoung\to \Laurent$ be a function on the set $\AllYoung$ of Young diagrams.
Suppose that for each $\ell\geq 1$ the equality
\[  F(P'\times Q')=\St_\ell(\gamma;P ; Q)\]
--- with the usual substitution \eqref{eq:gamma} for the variable $\gamma$ --- 
holds true for all choices of $\ell\geq 1$, $P$, $Q$ and $A\neq 0$ for which the multirectangular diagram $P'\times Q'$ is well-defined.
Then we say that $\St$ is the \emph{anisotropic Stanley polynomial} for $F$.
For a given function $F$, the corresponding Stanley polynomial, if exists, is unique
(in order to show this, one can adapt the corresponding part of the proof of \cite[Lemma 2.4]{DolegaFeraySniady2013}).
\end{definition}

\subsection{Isotropic Stanley polynomials}
\label{subsec:stanley-polynomials-isotropic}
We consider now the specialization of the concept of anisotropic Stanley polynomials 
to the special choice of $A=1$, $\gamma=0$. The resulting objects will be called 
\emph{isotropic Stanley polynomials}.

More specifically, 
let $\St'=(\St'_1,\St'_2,\dots)$ be a sequence of polynomials such that for each $\ell\geq 1$ 
\[\St'_\ell=\St'_\ell(p_1,\dots,p_\ell;q_1,\dots,q_\ell)=\St'_\ell(P; Q) \] 
is a polynomial in $2\ell$ variables and
\[ \St'_{\ell+1}(p_1,\dots,p_\ell,0;q_1,\dots,q_\ell,0)=\St'_\ell(p_1,\dots,p_\ell;q_1,\dots,q_\ell). \]
We assume furthermore that the degrees of the polynomials $\St'_1,\St'_2,\dots$ are uniformly bounded by some integer $d$;
we say then that the degree of $\St'$ is at most $d$.

\begin{definition}
\label{def:stanley-polynomial-iso}
Let $F\colon \AllYoung\to \Q$ be a function on the set $\AllYoung$ of Young diagrams.
Suppose that for each $\ell\geq 1$ the equality
\[  F(P \times Q)=\St'_\ell(P ; Q)\]
holds true for all choices of $\ell\geq 1$, $P$, $Q$ for which the multirectangular diagram $P\times Q$ is well-defined.
Then we say that $\St'$ is the \emph{isotropic Stanley polynomial} for $F$.
\end{definition}

\subsection{The isotropic case $\alpha=1$}

For $k\geq 2$ we denote by $T_k$, $S_k$ and $R_k$ 
the versions of the functionals $\Tfunct_k$, $\Sfunct_k$ and $\Rfunct_k$
which were specialized to the case $A=1$:
\begin{align}
\nonumber
T_k (\lambda) &:= (n-1) \sum_{\Box=(x,y)\in\lambda} \big( x-y \big)^{k-2} \in \Q, \\
\label{eq:what-is-isotropic-s}
S_k (\lambda) &:= 
(k-1) \iint_{(x,y)\in\lambda} (x-y)^{k-2}\, \mathrm{d}x\, \mathrm{d}y \in \Q, \\
\label{eq:what-is-isotropic-r}
R_k(\lambda) &:= (-1) \sum_{\sigma_1,\sigma_2}  (-1)^{|C(\sigma_2)|} N_{\sigma_1,\sigma_2}(\lambda)\in \Q,
\end{align}
for any Young diagram $\lambda$,
where the sum in \eqref{eq:what-is-isotropic-r} runs over the same set as in \eqref{eq:definition-free-cumulant}.

\subsection{Anisotropic vs isotropic}
\label{sec:anisotropic-isotropic}

\begin{lemma}
\label{lem:stanley-polunomial-isotropic-and-anisotropic-the-same}
For each bicolored graph $G$ 
the anisotropic Stanley polynomial for the function
\begin{equation}
\label{eq:re-signed-embed}
\lambda\mapsto(-1)^{|\V_\bullet(G)|}\ \Embed_G(\lambda)    
\end{equation}
exists and coincides with 
the isotropic Stanley polynomial for the function $\lambda\mapsto N_G(\lambda)$.
This polynomial is homogeneous of degree $|\V(G)|$.
When viewed as a polynomial in the variables $q_1,q_2,\dots$ with coefficients in
$\Q[p_1,p_2,\dots]$, this polynomial is homogeneous of degree $|\V_{\circ}(G)|$.
\end{lemma}
\begin{proof}
As pointed out in the proof of \cite[Lemma 2.4]{DolegaFeraySniady2013},
a slight variation of \cite[Lemma 3.9]{FeraySniady2011} shows that if $G$ is an arbitrary bicolored graph,
then for the function \eqref{eq:re-signed-embed}
the corresponding anisotropic Stanley polynomial exists and 
does not involve the variable $\gamma$. 
In particular, it 
coincides with 
the isotropic Stanley polynomial for the function $N_G$ from \cref{def:embeddings}.
\end{proof}

\begin{lemma}
\label{eq:S-is-for-Stanley}
For each integer $n\geq 2$,
the anisotropic Stanley polynomial for 
$\lambda\mapsto \Sfunct_n(\lambda)$ exists and 
coincides with the isotropic Stanley polynomial for the isotropic 
functional $\lambda\mapsto S_n(\lambda)$.
These polynomials are homogeneous of degree $n$.
\end{lemma}
\begin{proof}  
We use the notations from \cref{sec:multirectangular-coordinates}.
An elementary integration shows that
whenever $P'\times Q'\in\AllYoung$, then
\begin{multline}
\label{eq:sfunct-in-stanley-coordinates}
\Sfunct_n(P'\times Q')=\\
\frac{-1}{n} \sum_{i\geq 1}  
\Big[ 
\big(-(p_1+\dots+p_{i-1})\big)^{n}
-\big(-(p_1+\dots+p_{i})\big)^{n}-\\
-\big(q_i-(p_1+\dots+p_{i-1})\big)^n
+\big(q_i-(p_1+\dots+p_{i})\big)^n
 \Big]. 
\end{multline}
It is worth pointing out that 
the right-hand side is a polynomial which does not involve the variable $\gamma$.

An analogous calculation performed for $S_n(P\times Q)$ gives the same polynomial.
\end{proof}

\begin{lemma}
\label{eq:R-is-for-Ranley}
For each integer $n\geq 2$,
the anisotropic Stanley polynomial for 
$\lambda\mapsto \Rfunct_n(\lambda)$ exists and 
coincides with the isotropic Stanley polynomial for the isotropic 
functional $\lambda\mapsto R_n(\lambda)$.
These polynomials are homogeneous of degree $n$.
\end{lemma}
\begin{proof}
It is a direct consequence of \cref{lem:stanley-polunomial-isotropic-and-anisotropic-the-same}
and the definitions \eqref{eq:definition-free-cumulant}, \eqref{eq:what-is-isotropic-r}.
\end{proof}

\subsection{Discrete vs smooth}

\begin{proposition}
\label{prop:generate-the-same}
The filtered unital algebra $\Poly$ of $\alpha$-polynomial functions
(cf.~\cref{sec:polynomial-functions})
can be alternatively viewed as generated by the elements $\gamma,\Sfunct_2,\Sfunct_3,\dots$
with the degrees of the generators given by
\begin{equation}
\label{eq:filtration2}
\left\{
\begin{aligned}
\degg \gamma    &= 1,  \\  
\degg \Sfunct_n &= n \qquad \text{for $n\geq 2$}.
\end{aligned}
\right.
\end{equation}
\end{proposition}

The above result follows from the following lemma which shows that the passage from the algebraic base
$\Sfunct_2,\Sfunct_3,\dots$ to the algebraic base  $\Tfunct_2,\Tfunct_3,\dots$ (as well as 
passage in the opposite direction)
is given by \emph{linear} equations (with the coefficients in $\Q[\gamma]$)
with appropriate degree bounds.

\begin{lemma}
\label{eq:same-filtered-vector-space}
The following two $\Q[\gamma]$-modules are equal as filtered vector spaces over $\Q$:
\begin{itemize}
   \item the module spanned by $\Tfunct_2,\Tfunct_3,\dots$ with the filtration defined on the generators by
\eqref{eq:filtration};

   \item the module spanned by $\Sfunct_2,\Sfunct_3,\dots$ with the filtration defined on the generators by
\eqref{eq:filtration2}.
\end{itemize}
\end{lemma}
In other words, the lemma states that for each integer $d$ the following two linear spaces are equal:
\begin{multline}
\label{eq:filtrations-equal} 
 \operatorname{span}\left\{ \gamma^k \Tfunct_n : k\geq 0, \; n\geq 2, \; k+n \leq d \right\} = \\
\operatorname{span}\left\{ \gamma^k \Sfunct_n : k\geq 0, \; n\geq 2, \; k+n \leq d \right\}.
\end{multline}
\begin{proof}   
We start by expressing $\Sfunct_n$ in terms of the generators \eqref{eq:filtration}.

A single box $(x_0,y_0)\in\N^2$ of a Young diagram, when viewed as a subset of the plane, becomes  
the square
\[\{(x,y): x_0< x \leq x_0+1, \quad  y_0<y \leq y_0+1\} \subset \R^2.\]
The integral on the right-hand side of \eqref{eq:smooth-functional}
restricted to this box is given by:
\begin{multline}
\label{eq:integration-single-box}
(n-1) \int_{x_0}^{x_0+1} \left[ \int_{y_0}^{y_0+1}  (A x - A^{-1} y)^{n-2} \, \mathrm{d}y \right] \mathrm{d}x  = \\ 
\frac{-1}{n} \left[ 
\left(c+A-A^{-1}\right)^{n}
- \left(c+A\right)^{n}
- \left(c-A^{-1}\right)^{n}
+c^n
 \right] ,
\end{multline}
where on the right-hand side
\[ c:= \acontent(x_0,y_0)= Ax_0-A^{-1} y_0.\]

We shall view the right-hand side of \eqref{eq:integration-single-box} as a polynomial in the variable $c$
of the following form:
\[ \sum_{2\leq k\leq n+2} d_k\ (k-1)\ c^{k-2}\]
with the coefficients $d_2,\dots,d_{n+2}\in\Laurent$ given by
\[
d_k= 
\frac{\binom{n}{k-2}}{n (k-1)} 
\left[ 
- \left(A-A^{-1}\right)^{n+2-k}
+ A^{n+2-k}
+ \left(-A^{-1}\right)^{n+2-k}
- 0^{n+2-k}
 \right].\]
Each coefficient $d_k$ is a Laurent polynomial which is invariant under 
the automorphism 
\begin{equation}
\label{eq:automorphism}
A\leftrightarrow -\frac{1}{A};
\end{equation}
an automorphism which is given explicitly as
\[ \Laurent\ni \sum_{k\in\Z} f_k A^k \mapsto  \sum_{k\in\Z} f_k \left( - \frac{1}{A} \right)^k \in \Laurent.
\]
An elementary calculation based on the binomial formula shows that --- due to cancellations --- 
$d_k$ is a Laurent polynomial of degree at most $n-k$ for each $2\leq k\leq n$.
Furthermore, $d_{n+1}=d_{n+2}=0$ and $d_n=1$.

By comparing the dimensions it follows that the space of the Laurent polynomials
of degree at most $n-k$ which are invariant under the automorphism \eqref{eq:automorphism}
is spanned by $1,\gamma,\dots,\gamma^{n-k}$.
In this way we proved existence of a polynomial
$P_k\in\Q[\gamma]$ of degree at most $n-k$ with the property that 
\[ d_k = P_k(\gamma),\]
where on the right-hand side the usual substitution \eqref{eq:gamma} is applied.

As the integral over a Young diagram $\lambda\subset\R^2$ can be written as a sum of the integrals 
over the individual boxes, it follows immediately that the following equality of functions on $\AllYoung$
holds true
\begin{equation}
\label{eq:s-versus-t}
 \Sfunct_n= \sum_{2\leq k\leq n} P_k(\gamma)\ \Tfunct_k   
\end{equation}
with the polynomials $P_2,\dots,P_n$ given by the above construction.
\emph{This shows that the right-hand side of \eqref{eq:filtrations-equal} is a subset of its
left-hand side, as required.}

\bigskip

We will show that for each $n\geq 2$
there exist polynomials $Q_2,\dots,Q_{n}\in\Q[\gamma]$ with the property that
\begin{equation}
\label{eq:t-versus-s}
 \Tfunct_n= \sum_{2\leq k\leq n} Q_k(\gamma)\ \Sfunct_k
\end{equation}
and, furthermore, the degree of $Q_k$ is bounded from above by $n-k$.
Our proof will use induction with respect to the variable $n$.
Equation \eqref{eq:s-versus-t} can be written in the form
\[ \Tfunct_n = \Sfunct_n - \sum_{2 \leq k \leq n-1} P_k(\gamma)\ \Tfunct_k.
\]
The inductive assertion can be applied to each of the expressions $\Tfunct_2,\dots,\Tfunct_{n-1}$ on
the right-hand side. It follows that $\Tfunct_n$ can be written, as required, in the form \eqref{eq:t-versus-s}
with the proper bounds on the degrees of the polynomials $Q_2,\dots,Q_n$.
This concludes the proof of the inductive step.

Equation \eqref{eq:t-versus-s} and the degree bounds on the polynomials $Q_2,\dots,Q_n$ imply 
that 
\emph{the left-hand side of \eqref{eq:filtrations-equal} is a subset of its
right-hand side, as required.}
\end{proof}

\subsection{Yet another basis of $\Poly$: free cumulants}

\begin{proposition}
\label{prop:generate-the-same=free}
The filtered unital algebra $\Poly$ of $\alpha$-polynomial functions
can be alternatively viewed as generated by the elements $\gamma,\Rfunct_2,\Rfunct_3,\dots$
with the degrees of the generators given by
\begin{equation}
\label{eq:filtration3}
\left\{
\begin{aligned}
\degg \gamma    &= 1,  \\  
\degg \Rfunct_n &= n \qquad \text{for $n\geq 2$}.
\end{aligned}
\right.
\end{equation}
\end{proposition}
\begin{proof}
In the light of \cref{prop:generate-the-same} it is enough to show that
$\Sfunct_2,\Sfunct_3,\dots$ and $\Rfunct_2,\Rfunct_3,\dots$ generate the same 
filtered algebra with the usual choice of the degrees of the generators
\eqref{eq:filtration2} and \eqref{eq:filtration3}. More specifically, we will show that
for each $n\geq 2$:
\begin{itemize}
   \item $\Rfunct_n$ can be expressed as a polynomial 
$F(\Sfunct_2,\dots,\Sfunct_n)$ 
for some multivariate polynomial $F(x_2,\dots,x_n)$, 
   \item $\Sfunct_n$ can be expressed as a polynomial 
$G(\Rfunct_2,\dots,\Rfunct_n)$ 
for some multivariate polynomial $G(x_2,\dots,x_n)$, 
\end{itemize}
and that $F$ and $G$ are homogeneous polynomials of degree $n$,
where the degrees of the variables are specified by $\degg x_i=i$. 
In order to do this we shall study the relationship between the anisotropic
Stanley polynomials (\cref{subsec:stanley-polynomials})
and their isotropic counterparts (\cref{sec:anisotropic-isotropic}).

\medskip

In the following we concentrate on the problem of finding the polynomial $F$.
Each anisotropic Stanley polynomial determines uniquely the corresponding function on the set $\AllYoung$
of Young diagrams.
Therefore our problem can be reformulated as expressing 
\emph{the anisotropic Stanley polynomial for $\Rfunct_n$} as a polynomial (which is equal to 
our wanted polynomial $F$)
in terms of 
\emph{the anisotropic Stanley polynomial for $\Sfunct_2$},
\emph{the anisotropic Stanley polynomial for $\Sfunct_3$},\dots.

\cref{eq:R-is-for-Ranley} (respectively, \cref{eq:S-is-for-Stanley}) 
shows equality
between the anisotropic Stanley polynomial for $\Rfunct_n$ (respectively, $\Sfunct_n$)
and the isotropic Stanley polynomial for $R_n$
(respectively, for $S_n$).
Since Stanley polynomial for a given function on $\AllYoung$ (if exists) is unique,
our goal is equivalent to proving the existence of a multivariate polynomial $F$ 
with the property that
\[ R_n(\lambda)= F\big(S_2(\lambda),S_3(\lambda),\dots\big)\]
holds for every $\lambda\in\AllYoung$.

The latter polynomial is known to exist and its exact form is known \cite[Eq.~(15)]{DolegaFeraySniady2008}.
The latter formula also implies the required degree bound.

\medskip

The problem of finding the polynomial $G$ 
is analogous with the roles of the quantities $\Sfunct_n$ and $\Rfunct_n$ interchanged.
In the last step of the proof one should use \cite[Eq.~(14)]{DolegaFeraySniady2008} instead.
\end{proof}

\section{Degree bounds of Dołęga and F\'eray}
\label{sec:vk-scaling}

\subsection{Kerov--Lassalle polynomial}
\label{sec:KVpolynomial}
If $\mu=(\mu_1,\mu_2,\dots)$ is a partition which does not contain any parts equal to $1$, we define
the function $\Rfunct_\mu$ by a multiplicative extension of free cumulants:
\[ 
\Rfunct_\mu  := \prod_{k} \Rfunct_{\mu_k}. 
\]

Dołęga and F\'eray \cite{DoleegaFeray2014} studied the \emph{Kerov--Lassalle polynomial}, 
i.e., the expansion of the Jack character $\Ch_\pi$ in 
the linear basis $(\Rfunct_\mu)$ with the coefficients which a priori belong to 
the field $\Q(A)$ of rational functions in the variable $A$.
They proved \cite[Corollary 3.5]{DoleegaFeray2014} that each such a coefficient
$\left[ \Rfunct_\mu\right] \Ch_\pi$ is, in fact, a polynomial in the variable $\gamma$.

\subsection{Degree bounds of Dołęga and F\'eray}
In the following we will consider partitions $\mu$ and $\pi$ for which the corresponding coefficient 
\begin{equation}
\label{eq:KerovLassalle}
  \left[ \Rfunct_\mu\right]  \Ch_\pi\in \Q[\gamma]   
\end{equation}
of Kerov--Lassalle polynomial is non-zero and we denote by
\[ \ddeg:=\ddeg \big( [\Rfunct_\mu] \Ch_\pi\big) \]
the degree of this polynomial.
Dołęga and F\'eray also proved \cite[Proposition 3.7, Proposition 3.10]{DoleegaFeray2014} 
that 
\begin{align}
\label{eq:restriction-linear-combination-first}
|\mu| + \ddeg &\leq |\pi|+\ell(\pi), \\
\label{eq:restriction-linear-combination}
 |\mu|-2\ell(\mu) +\ddeg &\leq |\pi|-\ell(\pi).\\
\intertext{
By taking the mean of the above inequalities we obtain}
\label{eq:restriction-linear-combination-conclusion}
|\mu|-\ell(\mu)+\ddeg &\leq |\pi|. 
\end{align}

\subsection{Degree of the Jack characters}

\begin{theorem}
\label{thm:degree-of-Jack-character}
Let $\pi$ be an arbitrary partition.

Then $\Ch_\pi\in\Poly$ is an $\alpha$-polynomial function of degree at most $|\pi|+\ell(\pi)$.   
\end{theorem}
\begin{proof}
This is a direct consequence of \eqref{eq:restriction-linear-combination-first}
combined with \cref{prop:generate-the-same=free}.
\end{proof}

\begin{proposition}
\label{prop:filtration-by-characters}
The family
\begin{equation}
\label{eq:all-your-base-are-belong-to-us}
 \big\{ \gamma^d \Ch_\pi : d\geq 0, \text{$\pi$ is a partition} \}    
\end{equation}
is a linear basis  of $\Poly$.

The usual filtration on $\Poly$ can be equivalently defined by
setting the degrees of elements of this linear basis as
\[ \degg \gamma^d \Ch_\pi = d + |\pi|+\ell(\pi).\] 
\end{proposition}
\begin{proof}
The linear independence of \eqref{eq:all-your-base-are-belong-to-us}
was proved (in a wider generality) by Dołęga and F\'eray \cite[Proposition 2.9]{DoleegaFeray2014}.

\smallskip

Our goal is to show the equality between vector spaces:
\begin{multline*} \operatorname{span}\big\{ \gamma^d \Ch_\pi     : d\geq 0, \ d+ |\pi|+\ell(\pi) \leq n \} = \\
   \operatorname{span}\big\{ \gamma^d \Rfunct_\mu : d\geq 0,\ d+ |\mu|           \leq n \},
\end{multline*}
where on the left-hand side the span runs over partitions $\pi$,
while on the right-hand side the span runs over partitions $\mu$ which do not contain any part equal to $1$.

The inclusion $\subseteq$ is a direct consequence of \eqref{eq:restriction-linear-combination-first}.
On the other hand, the cardinality of the base of the 
left-hand side is equal to the cardinality of the generating set of the right-hand side;
thus these finite-dimensional linear spaces are indeed equal.
\end{proof}

\subsection{Vershik--Kerov scaling}
\label{sec:vk-concrete}

Vershik and Kerov \cite{VershikKerov1981a}
proved a special case of the following result  
(namely in the case $A=1$ which 
corresponds to the usual characters $\Ch_\pi^{A=1}$ of the symmetric groups).
In the setup of Jack characters it was proved (in a slightly different formulation) by Lassalle 
\cite[Proposition~2]{Lassalle2008a}.
Yet another proof, based on the results of Dołęga and F\'eray can be found in the early
version of the current work \cite[Proposition 3.4]{Sniady2016a}.

\begin{proposition}
\label{prop:vershik-kerov-jack-character}
Let $\pi$ be a partition and $m\geq 1$ be an integer.

Then    
\begin{equation} 
\label{eq:vershik-kerov-polynomial-jack-character}
\AllYoung \ni(\lambda_1,\dots,\lambda_m) \mapsto \Ch_\pi(\lambda_1,\dots,\lambda_m)\in\Laurent   
\end{equation}
is a polynomial of degree $|\pi|$; its homogeneous top-degree part is equal to
\[ A^{|\pi|-\ell(\pi)}\ p_{\pi}(\lambda_1,\dots,\lambda_m),  \]
where $p_\pi$ is the power-sum symmetric polynomial.
\end{proposition}

\subsection{Degrees of Laurent polynomials}

\begin{definition}
\label{def:degree-Laurent}
For an integer $d$ we will say that a \emph{Laurent polynomial
\[ f=\sum_{k\in\Z} f_k A^k\in \Laurent \] is of degree at most $d$} if
$f_k=0$ holds for each integer $k>d$. 
\end{definition}

\begin{proposition}
\label{prop:degree-laurent}
For any partition $\pi$ and any Young diagram $\lambda$ the evaluation
$\Ch_\pi(\lambda)\in\Laurent$ is a Laurent polynomial of degree at most $|\pi|-\ell(\pi)$.   
\end{proposition}
\begin{proof}
From the very definition of free cumulants \eqref{eq:definition-free-cumulant} it follows 
that for any $n\geq 2$ and any Young diagram $\lambda$, the evaluation
$\Rfunct_n(\lambda)\in\Laurent$ is a Laurent polynomial of degree at most $n-2$.
By multiplicativity it follows that 
$\Rfunct_\mu(\lambda)\in\Laurent$ is a Laurent polynomial of degree at most $|\mu|-2\ell(\mu)$.
On the other hand, $\gamma\in\Laurent$ is a Laurent polynomial of degree $1$.
The bound \eqref{eq:restriction-linear-combination} concludes the proof.
\end{proof}

\section{The algebras $\functions$ and $\functions_\row$ of row functions}
\label{sec:row-functions}

\subsection{Row functions}

\begin{definition}
\label{def:row-functions}
Let a sequence (indexed by $r\geq 0$) of symmetric functions
$f_{r}:\N_0^r \rightarrow \Laurent$ 
be given, where
\[\N_0=\{0,1,2,\dots\}.\]
We assume that:
\begin{itemize}
   \item if\/ $0\in \{x_1,\dots,x_r\}$ then $f_{r}(x_1,\dots,x_r)=0$,

   \item $f_{r}=0$ except for finitely many values of $r$,

   \item there exists some integer $d\geq 0$ with the property that for all $r\geq 0$ and all $x_1,\dots,x_r\in\N_0$,
            the evaluation $f_r(x_1,\dots,x_r)\in\Laurent$ is a Laurent polynomial of degree at most $d-2r$.
\end{itemize}

We define a function $F\colon \AllYoung\rightarrow\Laurent$ given by
\begin{equation}
\label{eq:row-functions}
 F(\lambda)  := 
\sum_{r\geq 0} \sum_{i_1 < \dots < i_r} f_{r}(\lambda_{i_1},\dots,\lambda_{i_r}).   
\end{equation}
We will say that $F$ is a \emph{row function of degree at most $d$} 
and that $\left(f_r\right)$ is the \emph{kernel} of $F$.
The set of such row functions will be denoted by $\functions$.
\end{definition}

\subsection{Filtration on $\functions$}
\label{sec:filtration-on-functions}

\begin{proposition}
\label{lem:filtered-algebra}
The set $\functions$ of row functions equipped with the pointwise product and pointwise addition forms a
filtered algebra.
\end{proposition}
\begin{proof}
A product of two terms appearing on the right-hand side of \eqref{eq:row-functions}
is again of the same form:
\begin{multline}
\label{eq:product-convolution-kernels}
 \sum_{i_1 < \dots  < i_r} f_{r}(\lambda_{i_1},\dots,\lambda_{i_r}) \cdot
 \sum_{j_1 < \dots  < j_s} g_{s}(\lambda_{j_1},\dots,\lambda_{j_s})  \\ 
\shoveleft{=
 \sum_{0\leq t\leq r+s} \;
\sum_{k_1 < \dots  < k_t}} \\ 
\underbrace{\sum_{\substack{ i_1 < \dots  < i_r \\ j_1 < \dots  < j_s \\ 
\{k_1,\dots,k_t\}=\{i_1,\dots,i_r\} \cup \{j_1,\dots,j_s\} }}
f_{r}(\lambda_{i_1},\dots,\lambda_{i_r}) \ 
g_{s}(\lambda_{j_1},\dots,\lambda_{j_s})}_{h_{t}(\lambda_{k_1},\dots,\lambda_{k_t}):=}= \\
 \sum_{0\leq t\leq r+s} \;
\sum_{k_1 < \dots  < k_t} 
h_t(\lambda_{k_1},\dots,\lambda_{k_t})
\end{multline}
which shows that $\functions$ forms an algebra.

\bigskip

Assume that the two factors on the left-hand of \eqref{eq:product-convolution-kernels}
are row functions of degree at most, respectively, $d$ and $e$.
It follows that each value of $f_r$ is a Laurent polynomial of degree at most $d-2r$
and each value of $g_s$ is a Laurent polynomial of degree at most $e-2s$. 
Thus each value of $h_t$, the kernel defined by the curly bracket 
on the right-hand side of \eqref{eq:product-convolution-kernels},
is a Laurent polynomial of degree at most $d+e-2(r+s)\leq d+e-2t$ which shows that
the product is a row function of degree at most $d+e$, which concludes the proof that
$\functions$ is a filtered algebra.
\end{proof}

\begin{proposition}
\label{lem:polynimial-are-row}
Let $F\in\Poly$ be an $\alpha$-polynomial function of degree $d$.

Then $F\in\functions$ is also a row function of degree at most $d$.
\end{proposition}
\begin{proof}
By \cref{lem:filtered-algebra} it is enough to
prove the claim for the generators \eqref{eq:filtration}, 
i.e., to show for each $n\geq 2$ that $\Tfunct_n$ is a row function of degree at most $n$  
and that $\gamma$ is a row function of degree $1$. 
We will do it in the following.

\bigskip

For a Young diagram $\lambda$ we denote by $\lambda^T=(\lambda^T_1,\lambda^T_2,\dots)\in\AllYoung$ the transposed diagram.
The binomial formula implies that
\begin{multline} 
\label{eq:tfunct-as-a-row-function}
\Tfunct_n(\lambda) = 
(n-1)
\sum_{x\geq 1} \sum_{1\leq y\leq \lambda^T_x } \left( Ax - A^{-1} y\right)^{n-2} =\\ 
(n-1) \sum_{q\geq 0}  A^{n-2-2q} \binom{n-2}{q}   (-1)^q
\sum_{x\geq 1} x^{n-2-q} \sum_{1\leq y\leq \lambda^T_x}      y^q= \\ 
(n-1) \sum_{u\geq 1}  A^{n-2u}     \binom{n-2}{u-1}   (-1)^{u-1}
\sum_{x\geq 1} x^{n-1-u} \underbrace{\sum_{1\leq y\leq \lambda^T_x}  y^{u-1}}_{(\spadesuit)},
\end{multline}
where the last equality follows from the change of variables $u:=q+1$.  % %  q+1=u,  q=u-1

The expression $(\spadesuit)$ marked above by the curly bracket, namely
\[ \N_0\ni s \mapsto \sum_{1\leq y\leq s} y^{u-1}, \] 
is a polynomial function of degree $u$, thus it can be written as a 
linear combination (with rational coefficients) of the family of polynomials 
$\N_0\ni s\mapsto \binom{s}{r}$ indexed by $r\in\{0,1,\dots,u\}$.
Notice that $\lambda^T_x$ is the number of rows of $\lambda$ which are bigger or equal than $x$,
thus $(\spadesuit)$ is a linear combination (with rational coefficients) of
\begin{equation}
\label{eq:binomials-form-a-basis} 
\binom{\lambda^T_x}{r} = \sum_{i_1 < \dots <  i_r}  [\lambda_{i_1} \geq x] \cdots [\lambda_{i_r} \geq x]
\end{equation}
over $r\in\{0,1,\dots,u\}$.
This shows that $\Tfunct_n$ is a row function.

\smallskip

Let $f_0,f_1,\dots$ be the corresponding kernel.
Equation \eqref{eq:tfunct-as-a-row-function} shows that each value 
$f_r(x_1,\dots,x_r)$ is a linear combination (with rational coefficients) of the expressions
\[  A^{n-2u}\ [x_1 \geq x] \cdots [x_r \geq x] \in\Laurent\]
over $u\geq 1$, over $x\geq 1$, and $r\leq u$. This Laurent polynomial is degree at most
$n-2u\leq n-2r$, which shows that $\Tfunct_n$ is a row function of degree at most $n$, as required.

\bigskip

We define
\[ f_r = \begin{cases}
            \gamma & \text{if $r=0$}, \\ 
            0      & \text{if $r\geq 1$}.
         \end{cases}
\]
Clearly, the corresponding row function $F$ fulfills $F(\lambda)=\gamma$ for any $\lambda\in\AllYoung$.
This shows that $\gamma$ is a row function of degree at most $1$, as required.
\end{proof}

\subsection{Separate product of row functions}

The set of row functions can be equipped with another product, which we will call
\emph{the separate product}.  
It is defined on the linear basis by declaring
\begin{multline*}
 \sum_{i_1 < \dots  < i_r} f(\lambda_{i_1},\dots,\lambda_{i_r}) \row
 \sum_{j_1 < \dots  < j_s} g(\lambda_{j_1},\dots,\lambda_{j_s}):=  \\ 
{
\sum_{k_1 < \dots  < k_{r+s}}} 
\underbrace{\sum_{\substack{ i_1 < \dots  < i_r \\ j_1 < \dots  < j_s \\ 
\{k_1,\dots,k_{r+s}\}=\{i_1,\dots,i_r\} \sqcup \{j_1,\dots,j_s\} }}
f(\lambda_{i_1},\dots,\lambda_{i_r}) 
g(\lambda_{j_1},\dots,\lambda_{j_s})}_{h_{r+s}(\lambda_{k_1},\dots,\lambda_{k_{r+m}})};
\end{multline*}
we extend this definition by bilinearity to general row functions.
This corresponds to selecting in \eqref{eq:product-convolution-kernels} only 
the summand for which $t=r+s$.
This product is well-defined since the kernel is
uniquely determined by the row function.

We define $\functions_\row$ as the linear space of row functions equipped with 
\emph{the separate product} $\row$.  
It is a very simple exercise to check that the algebra $\functions_\row$ 
equipped with the notion of degree from \cref{def:row-functions} becomes a 
\emph{filtered algebra}.

\subsection{Summary: four filtered algebras. Cumulants}
So far we have introduced four filtered algebras of functions on $\AllYoung$.
They can be summarized by the following commutative diagram.
\begin{equation}
\label{eq:commutative-diagram-B2}
\begin{tikzpicture}[node distance=3cm,auto, baseline=-14ex]
  \node (A) {$\Poly_\disjoint$};
  \node (B) [right of=A] {$\Poly$};
  \node (B1) [right of=B,node distance=1.2cm] {};
  \node (B2) [below of=B1,node distance=1.2cm] {$\functions$};

  \node (C) [below of=B2]{$\functions_\row$};

  \draw[->] (A) to node {$\id$} (B);
  \draw[->] (B) to node {$\id$} (B2);

  \draw[->] (B2) to node {$\id$} (C);
  \draw[->] (A) to node [swap]{$\id$} (C);
\end{tikzpicture}
\end{equation}
Each of the arrows is a unital linear map given by the identity (inclusion).

Each of the arrows is compatible with the corresponding filtrations in the sense that
the degree of the image of any element $x$ is bounded from above by the degree of $x$ itself
(for the horizontal and the vertical arrow this follows from the fact
that the corresponding pairs of algebras are isomorphic as \emph{filtered vector spaces};
for both diagonal arrows this corresponds to \cref{lem:polynimial-are-row}).

The short diagonal arrow is an inclusion of algebras.

For some arrows in this diagram we will investigate the corresponding cumulants.
We recall that for 
the horizontal arrow $\mathbb{E}=\id \colon \Poly_\disjoint \to \Poly$ the corresponding cumulant
is denoted by $\kumuDisjointPoint$.
For the long diagonal arrow $\mathbb{E}=\id \colon \Poly_\disjoint \to \functions_\row$ 
the corresponding cumulant
will be denoted by $\kumuDisjointRow$.
For the vertical arrow $\mathbb{E}=\id \colon \functions \to \functions_\row$ 
the corresponding cumulant will be denoted by $\kumuPointRow$.

\section{How to show that an $\alpha$-polynomial function is of small degree?}
\label{sec:how-to-prove-small}

\subsection{The difference operator}

\begin{definition}
\label{def:difference}
If $F=F(\lambda_1,\dots,\lambda_\ell)$ is a function of $\ell$ arguments and $1\leq j\leq \ell$, 
we define a new function $\Delta_{\lambda_j} F$ by
\begin{multline*}
\left( \Delta_{\lambda_j} F \right) (\lambda_1,\dots,\lambda_\ell):= \\ 
F(\lambda_1,\dots,\lambda_{j-1},\lambda_j+1,\lambda_{j+1},\dots,\lambda_\ell)-
F(\lambda_1,\dots,\lambda_\ell).
\end{multline*}
We call $\Delta_{\lambda_j}$ a \emph{difference operator}.
\end{definition}

\subsection{Extension of the domain of functions on $\AllYoung$}
\label{sec:extension}

Let $F$ be a function on the set of Young diagrams.
Such a function can be viewed as a function $F(\lambda_1,\dots,\lambda_\ell)$ defined 
for all non-negative integers $\lambda_1\geq \dots\geq \lambda_\ell$.
We will extend its domain, as follows.

\begin{definition} 
\label{def:extension}
If $(\xi_1,\dots,\xi_\ell)$ is an arbitrary sequence of non-negative integers, we denote
\[ F^{\sym}(\xi_1,\dots,\xi_\ell):= F(\lambda_1,\dots,\lambda_\ell), \]
where $(\lambda_1,\dots,\lambda_\ell)\in \AllYoung$ is the sequence $(\xi_1,\dots,\xi_\ell)$ 
sorted in the reverse order $\lambda_1\geq \dots \geq \lambda_\ell$.
In this way $F^{\sym}(\xi_1,\dots,\xi_\ell)$ is a symmetric function of its arguments.
\end{definition}

Note that the definition \eqref{eq:row-functions} of a row function 
does not require any modifications in order to give rise to such an extension.
For this reason, if $F$ is a row function, we will identify it with its extension
$F^{\sym}$.

\subsection{The difference operator vanishes on elements of small degree}

\newcommand{\variableK}{k}

\begin{lemma}
\label{coro:small-degree-killed}   
Let $d\geq 1$ be an integer and assume that $F\in\functions$ is of degree at most $d-1$.

Then for each integer $\variableK\geq 0$ and each Young diagram $\lambda=(\lambda_1,\lambda_2,\dots)$ 
\[ [A^{d-2\variableK}] \Delta_{\lambda_1} \cdots \Delta_{\lambda_\variableK} F^{\sym}(\lambda_1,\lambda_2,\dots)  =0.\]
\end{lemma}
\begin{proof}
We know that
$F$ is a sum of the functions of the form
\[
 \sum_{i_1 < \dots < i_l} f_{l}(\lambda_{i_1},\dots,\lambda_{i_l}),   
\]
over $l\geq 0$ and each value of $f_l$ is a Laurent polynomial of degree at most $d-1-2l$.

Clearly, 
\[ \Delta_{\lambda_1} \dots \Delta_{\lambda_\variableK} f_{l}(\lambda_{i_1},\dots,\lambda_{i_l}) = 0 
   \qquad \text{if } \{1,\dots,\variableK\}\not\subseteq \{i_1,\dots,i_l\}
\]
thus  
\begin{equation}
\label{eq:tralala}
[A^{d-2\variableK}] \Delta_{\lambda_1} \dots \Delta_{\lambda_\variableK} 
\sum_{i_1 < \dots < i_l} f_{l}(\lambda_{i_1},\dots,\lambda_{i_l})       
\end{equation}
vanishes if $l<\variableK$.

On the other hand, for $l\geq \variableK$, the expression $f_l(\lambda)$ 
is a Laurent polynomial of degree at most $d-1-2l\leq d-1-2\variableK $
thus \eqref{eq:tralala} vanishes as well.
\end{proof}

\subsection{What happens to an $\alpha$-polynomial function if we view it as a row function?}
\label{sec:top-degree-row-function}
\begin{definition}
Let $d\geq 0$ be an integer and 
let $F$ given by \eqref{eq:row-functions} be a row function of degree at most $d$. 
We define its \emph{top-degree part} as: 
\begin{equation}
\label{eq:top-degree-row-function}   
 F^{\ttop}(\lambda)  := 
\sum_{r \geq 0} 
A^{d-2r} 
\sum_{i_1 < \dots < i_r} 
\left[A^{d-2r} \right] f_{r}(\lambda_{i_1},\dots,\lambda_{i_r}). 
\end{equation}
We will also say that the summand corresponding to a specified value of $r$
(i.e., the $r$-fold sum over the rows) \emph{has rank $r$}.
\end{definition}

\begin{lemma}
\label{lem:minimal-rank}
Let $r\geq 0$ and $d\geq 2r$ be integers and 
let $\poly(c_1,\dots,c_r)$ be a symmetric polynomial in its $r$ arguments with the coefficients 
in the polynomial ring $\Q[\gamma]$. We assume that $p$, viewed as a polynomial in $\gamma,c_1,\dots,c_r$,
is a homogeneous polynomial of degree $d-2r$.

We consider the row function of degree at most $d$ given by
\[
F(\lambda)= \sum_{\Box_1,\dots,\Box_r\in \lambda}
\poly(c_1,\dots,c_r),  
\]
where
\[c_1:=\acontent(\Box_1),\quad \dots, \quad c_r:=\acontent(\Box_r).
\]

Then each non-zero summand in \eqref{eq:top-degree-row-function} 
for the top-degree part of $F$ has rank at least $r$;
the summand with the rank equal to $r$ is given by
\begin{equation}
\label{eq:minimal-rank}
\lambda\mapsto A^{d-2r} r! \sum_{i_1<\cdots<i_r} 
\sum_{1\leq x_1 \leq \lambda_{i_1}} \cdots \sum_{1\leq x_r \leq \lambda_{i_r}}
 \poly(x_1,\dots,x_r) \bigg|_{\gamma=-1},
\end{equation}
where on the right-hand side we consider the evaluation of the polynomial $\poly$ for $\gamma=-1$.
\end{lemma}
\begin{proof}
We start with a general investigation of the top-degree part of various row functions.
The proof of \cref{lem:polynimial-are-row} shows that
\[ \gamma^{\ttop}=-A \]
consists of a single summand of rank $0$.

\medskip

The extraction of the top-degree part of the row function \eqref{eq:tfunct-as-a-row-function}
corresponds to the restriction to the summand $r=u\geq 1$ 
(with the notations of \eqref{eq:binomials-form-a-basis}). 
For this reason $\Tfunct_n^{\ttop}(\lambda)$ involves only the summands with the rank at least $1$.
Furthermore, the term of rank $1$ is given explicitly in the following expansion:
\[\Tfunct_n^{\ttop}(\lambda) = A^{n-2} \sum_i \sum_{1\leq x\leq \lambda_i} (n-1) x^{n-2} +
 (\text{summands of rank at least $2$}).
\]

\medskip

We shall revisit \eqref{eq:product-convolution-kernels} in order to investigate the top-degree part
of a product of two row functions.
The summands on the right-hand side do not contribute to the top-degree part unless $t=r+s$ and 
\[ \{k_1,\dots,k_t\}=\{i_1,\dots,i_r\} \sqcup \{j_1,\dots,j_s\}  \]
is a decomposition into disjoint sets.
This shows that the top-degree part of a product involves only the summands with the rank at least 
the sum of the ranks of the original factors.
Furthermore, the top-degree summand of this minimal rank is given very explicitly.

\bigskip

We come back to the proof of the Lemma.
Assume for simplicity that the polynomial $\poly$ is a monomial.
In this case $F$ is --- up to simple numerical factors --- a product of some power of
$\gamma$ and of exactly $r$ factors of the form $\Tfunct_n$ over $n\geq 2$. 
The above discussion shows that $F^{\ttop}$ involves only the summands of rank 
at least $r$, and gives a concrete formula for the summand of rank $r$.
It is easy to check that it is the formula \eqref{eq:minimal-rank}
in which the monomial $\poly$ has been replaced by its symmetrization.

By linearity, this result remains true for a general polynomial $\poly$.
In particular, if $\poly$ is already symmetric, formula \eqref{eq:minimal-rank} holds true
without modifications.  
\end{proof}

\subsection{Multivariate polynomials having lots of zeros}
\label{sec:preparation-proof-key-lemma-start}

The final ingredient in the proof of \cref{lem:keylemma}
is the following result which shows that
if a multivariate polynomial has a specific set of zeros, it must be identically equal to zero.

\begin{lemma}
\label{lem:lots-of-zeros}
Let $k\geq 0$ and $d\geq 0$ be integers.

\begin{itemize}
\item
Let $\poly(x_1,\dots,x_k)\in\Q[x_1,\dots,x_k]$ be a polynomial of degree at most~$d$.
Assume that
\begin{equation}
\label{eq:let-be-zero}
\poly(x_1,\dots,x_k)=0 
\end{equation}
holds true for all integers $x_1,\dots,x_k\geq 1$  such that 
\begin{equation}
\label{eq:sum-of-xes}
x_1+\dots+x_k\leq d+k.
\end{equation}

Then $\poly=0$.

\item
Let $\poly(x_1,\dots,x_k)\in\Q[x_1,\dots,x_k]$ be a \emph{symmetric} polynomial of degree at most $d$.
Assume that \eqref{eq:let-be-zero} holds true
for all integers $x_1\geq \dots \geq x_k\geq 1$ such that \eqref{eq:sum-of-xes} holds true.

Then $\poly=0$.
\end{itemize}
\end{lemma}
\begin{proof}
We will show the first part of the claim by induction over $k$.

\emph{The case $k=0$.} In this extreme case the empty sequence fulfills the assumption
\eqref{eq:sum-of-xes}; the resulting \eqref{eq:let-be-zero} gives the desired claim. 

\bigskip

\emph{The case $k=1$.} It follows that $\poly(x_1)\in\Q[x_1]$ is a polynomial of degree at most $d$
which has at least $d+1$ zeros; it follows that $\poly=0$.

\bigskip

\emph{We consider the case $k\geq 2$ and we assume that the first part of the lemma is true 
for $k':=k-1$.}
The polynomial $\poly$ can be written in the form
\[
 \poly = \sum_{0\leq r\leq d} 
       \poly_{r}(x_1,\dots,x_{k-1})\ \underbrace{(x_k-1) (x_k-2) \dots (x_k-r)}_{\text{$r$ factors}}, 
\] 
where $\poly_{r}\in\Q[x_1,\dots,x_{k-1}]$ is a polynomial of degree at most $d-r$.

We will show by a nested induction over the variable $r$ that $\poly_r=0$ for each $0\leq r \leq d$.
Assume that $\poly_l=0$ for each $l<r$. It follows that
\[  \poly(x_1,\dots,x_{k-1},r+1)=r!\ \poly_r(x_1,\dots,x_{k-1}),\]
thus 
\[ \poly_r(x_1,\dots,x_{k-1})=0 \]
holds true for all integers $x_1,\dots,x_{k-1}\geq 1$ such that 
\[x_1+\dots+x_{k-1}\leq (d-r)+(k-1).\]
It follows that $\poly_r$ fulfills the condition \eqref{eq:let-be-zero} for $d':=d-r$ and $k':=k-1$ thus 
the inductive hypothesis (with respect to the variable $k$) can be applied. 
It follows that $\poly_r=0$.
This concludes the proof of the inductive step over the variable $r$.

\bigskip

The second part of the lemma is a direct consequence of the first part.
\end{proof}

\subsection{The key tool}

\newcommand{\rrr}{r}

The assumptions of the following theorem have been modeled after the properties of the Jack characters;
in particular $F:=\Ch_\pi$ fulfills the assumptions 
--- except for the assumption \ref{zero:topdegree} ---
for $n:=|\pi|$ and $\rrr:=\ell(\pi)$.

\begin{theorem}[The key tool]
\label{lem:keylemma}
Let integers $n\geq 1$ and $\rrr\geq 1$ be given.
Assume that: 
\begin{enumerate}[label=(Z\arabic*)]
 \item
\label{zero:initial-bound}
$F\in\Poly$ is of degree at most $n+\rrr$;

 \item \label{zero:topdegree} 
we assume that
for each $m\geq 1$ the polynomial in $m$ variables
\[ \AllYoung \ni(\lambda_1,\dots,\lambda_m) \mapsto F(\lambda_1,\dots,\lambda_m) \]
is of degree at most $n-1$;

 \item 
\label{zero:vanishing}
the equality
\begin{equation}
\label{eq:top-key-equation} 
[A^{n+\rrr-2k}] \Delta_{\lambda_1} \cdots \Delta_{\lambda_k} F^{\sym}(\lambda_1,\dots,\lambda_k) = 0  
\end{equation}
holds true for the following values of\/ $k$ and $\lambda$:
\begin{itemize}
 \item $k=\rrr$ and $\lambda=(\lambda_1,\dots,\lambda_\rrr)\in\AllYoung$ with at most $\rrr$ rows is such that
$|\lambda|\leq n+\rrr-2k-1$;

 \item $k>\rrr$ and $\lambda=(\lambda_1,\dots,\lambda_k)\in\AllYoung$ with at most $k$ rows is such that
$|\lambda|\leq n+\rrr-2k$;
\end{itemize}

 \item 
\label{zero:laurent}
for each $\lambda\in\AllYoung$, the Laurent polynomial $F(\lambda)\in\Laurent$ 
is of degree at most $n-\rrr+1$.

\end{enumerate}

\medskip

Then $F\in\Poly$ is of degree at most $n+\rrr-1$. 

\bigskip

\emph{Alternative version:} the result remains valid 
for all integers $n\geq 0$ and $\rrr\geq 1$ 
if the assumption \ref{zero:topdegree} is removed and  
the condition \ref{zero:vanishing} is replaced by the following one:
\begin{enumerate}[label=(Z3a)]
 \item 
\label{zero:vanishing-B}
the equality \eqref{eq:top-key-equation} 
holds true for all $k\geq \rrr$ and $\lambda=(\lambda_1,\dots,\lambda_k)\in\AllYoung$ with at most $k$ rows
such that $|\lambda|\leq n+\rrr-2k$.

\end{enumerate}

\end{theorem}
\begin{proof}
The function $F\in\Poly$ can be written in the form
\begin{equation}
\label{eq:polynomial}
F(\lambda)= \sum_{k\geq 0}
\underbrace{\sum_{\Box_1,\dots,\Box_k\in
\lambda}
\poly_k(c_1,\dots,c_k)}_{F_k(\lambda)},   
\end{equation}
where
$c_i:=\acontent(\Box_i)$ 
and where $\poly_k$ is a symmetric polynomial in its $k$ arguments with the coefficients 
in the polynomial ring $\Q[\gamma]$.
Furthermore, the degree bound \ref{zero:initial-bound} implies that 
$\poly_k$ --- this time viewed as a polynomial in $k+1$ variables: $\gamma, c_1,\dots,c_k$ --- 
is a polynomial of degree at most $n+\rrr-2k$.
Let $\poly_k^{\ttop}$ denote its homogeneous part of degree $n+\rrr-2k$.

\bigskip

\emph{The statement of the theorem would follow if we can show that $\poly_k$ is, in fact, of degree at most
$n+\rrr-2k-1$ or, equivalently, $\poly_k^{\ttop}=0$.
We will show this claim by induction over $k\geq 0$; 
assume that $\poly_{m}$ is of degree at most
$n+\rrr-2m-1$ for each $m<k$.}

The curly bracket in \eqref{eq:polynomial} serves as the definition of the functions 
$F_0,F_1,\dots$ on the set $\AllYoung$ of Young diagrams. In the following we will investigate the quantity
--- which is analogous to \eqref{eq:top-key-equation} --- given by
\begin{equation}
\label{eq:mysterious-laurent}
 [A^{n+\rrr-2k}] \Delta_{\lambda_1} \cdots \Delta_{\lambda_k} F_m^{\sym}(\lambda_1,\dots,\lambda_k)  
\end{equation}
for $(\lambda_1,\dots,\lambda_k)\in\AllYoung$ and various choices of the variable $m\geq 0$.

\begin{itemize}[itemsep=2ex]
\item \emph{The case $m>k$.}
Firstly, observe that $\gamma$ as well as $\acontent(\Box)$ (for any box $\Box\in\N^2$), 
viewed as Laurent polynomials in the variable $A$, are of degree at most $1$, thus
$F_m(\lambda)$ is a Laurent polynomial of degree 
bounded from above by the degree of the polynomial $p_m$ which is
at most $n+\rrr-2m<n+\rrr-2k$.
It follows that for $m>k$ we have that 
\[ [A^{n+\rrr-2k}]  F_m^{\sym}(\lambda_1,\dots,\lambda_k) = 0 \]
hence, a fortiori,
the quantity 
\eqref{eq:mysterious-laurent} vanishes as well.

\item \emph{The case $m<k$.} 
From the inductive hypothesis, $F_m\in\Poly$ is an $\alpha$-polynomial function of degree at most $n+\rrr-1$.
We apply \cref{coro:small-degree-killed} for $d:=n+\rrr$.
In this way we proved that for $m<k$ the quantity \eqref{eq:mysterious-laurent} vanishes.

\item 
\emph{The case $m=k$.}
We will revisit the case $m<k$ considered above and discuss the changes in the reasoning.
We study now the function $F_k$. 
As the induction hypothesis cannot be applied,
the assumption of \cref{coro:small-degree-killed} is not satisfied for 
$F:=F_k$ and $d:=n+\rrr$
and we have to revisit its proof. We shall do it in the following.

The upper bound on the degree of the Laurent polynomial $f_l(\lambda)$ is weaker, given by 
$n+\rrr-2l$.
One can easily see that the only case in which \eqref{eq:tralala} could possibly be non-zero is for $l=k=m$,
thus
\begin{equation}
\label{eq:po-bezsennej-nocy-1}
\eqref{eq:mysterious-laurent} = 
\Delta_{\lambda_1} \cdots \Delta_{\lambda_k} 
\underbrace{[A^{n+\rrr-2k}]
\sum_{i_1 < \dots < i_k} f_{k}(\lambda_{i_1},\dots,\lambda_{i_k})}_{(\diamondsuit)}.
\end{equation}
Clearly, the expression $(\diamondsuit)$ is directly related to the top-degree part of $F_k\in\functions$ of rank $k$.
The latter can be computed explicitly by \cref{lem:minimal-rank} (applied for $r:=k$).
Thus
\begin{multline*}  
\eqref{eq:po-bezsennej-nocy-1}= \\ 
\Delta_{\lambda_1} \cdots \Delta_{\lambda_k}
k!
\sum_{i_1 < \dots < i_k} 
\sum_{1\leq x_1\leq \lambda_{i_1}} \cdots \sum_{1\leq x_k\leq \lambda_{i_k}}
\left. \poly^{\ttop}_{k}(x_{i_1},\dots,x_{i_k}) \right|_{\gamma:=-1} =
 \\
k!\; \left. \poly_k^{\ttop}(\lambda_1+1,\dots,\lambda_k+1) \right|_{\gamma:=-1}.
\end{multline*}
\end{itemize}

\bigskip

\emph{This finishes our discussion of the quantity \eqref{eq:mysterious-laurent} for various choices of the variable $m$.
The conclusion is that
\begin{multline}
 \label{eq:difference-gives-something}  
 [A^{n+\rrr-2k}] \Delta_{\lambda_1} \cdots \Delta_{\lambda_k} F^{\sym}(\lambda_1,\dots,\lambda_k) = \\
k!\; \left. \poly_k^{\ttop}(\lambda_1+1,\dots,\lambda_k+1) \right|_{\gamma:=-1}
\end{multline}
holds true for an arbitrary Young diagram $(\lambda_1,\dots,\lambda_k)\in\AllYoung$ with at most $k$ rows.
We will use this equality to finish the proof of the inductive step over the variable $k$.}

\medskip

Notice that the polynomial 
\begin{equation} 
\label{eq:top-polynomial-top-model}
\left. \poly_k^{\ttop}(c_1,\dots,c_k) \right|_{\gamma:=-1}
\end{equation}
in which we used the substitution $\gamma:=-1$ is an 
(inhomogeneous) symmetric polynomial in the indeterminates $c_1,\dots,c_k$ of degree at most $n+\rrr-2k$.
\emph{In order to achieve our ultimate goal and show that 
the homogeneous polynomial $\poly_k^{\ttop}$ is equal to zero 
it is enough to show that the inhomogeneous
polynomial \eqref{eq:top-polynomial-top-model} is equal to zero.
We shall do it in the following.}

\begin{itemize}[itemsep=2ex]
   \item 
\emph{Firstly, consider the case $k<\rrr$.} Assumption \ref{zero:laurent} 
on the degrees of Laurent polynomials implies 
that 
\[ [A^{n+\rrr-2k}]  F^{\sym}(\lambda_1,\dots,\lambda_k)=0 \]
holds true for any Young diagram with at most $k$ rows. Thus
the left-hand side of \eqref{eq:difference-gives-something} is constantly equal to zero.  
\cref{lem:lots-of-zeros} can be applied to the polynomial \eqref{eq:top-polynomial-top-model}; 
it follows that \eqref{eq:top-polynomial-top-model}
is the zero polynomial as required.

\item
\emph{Secondly, consider the case $k>\rrr$.} Assumption \ref{zero:vanishing} implies that the left-hand side of 
\eqref{eq:difference-gives-something} is equal to zero for $|\lambda|\leq n+\rrr-2 k$,
therefore 
$\left. \poly_k^{\ttop}(c_1,\dots,c_k) \right|_{\gamma:=-1}=0$ 
for all integers $c_1,\dots,c_k$ such that $c_1\geq \cdots \geq c_k \geq 1$ and $c_1+\dots+c_k\leq n+\rrr-k$.
Thus \cref{lem:lots-of-zeros} implies that \eqref{eq:top-polynomial-top-model}
is the zero polynomial, as required.

The same proof works for the alternative assumption \ref{zero:vanishing-B} in the case $k\geq \rrr$.

\item
\emph{Finally, consider the case $k=\rrr$.}
Note that this case for the alternative assumption \ref{zero:vanishing-B} was already considered above and
the following discussion is not applicable.

We consider the set of Young diagrams $\lambda=(\lambda_1,\dots,\lambda_k)$ with the property that 
$\lambda_1>\cdots>\lambda_k$.
For any Young diagram in this set
\begin{equation}
\label{eq:difference-hard-and-soft}
 \Delta_{\lambda_1} \cdots \Delta_{\lambda_k} F^{\sym}(\lambda_1,\dots,\lambda_k) =
\Delta_{\lambda_1} \cdots \Delta_{\lambda_k} F(\lambda_1,\dots,\lambda_k)
\end{equation}
and the extension of the domain of $F$ by symmetrization is not necessary.
We can view $F$ as a polynomial 
in the indeterminates $\lambda_1,\dots,\lambda_k$. 
One can easily show that if two polynomials in the variables $\lambda_1,\dots,\lambda_k$ 
coincide on the above set of Young diagrams
then they must be equal; \eqref{eq:difference-gives-something} and \eqref{eq:difference-hard-and-soft}
imply therefore the following equality between
polynomials:
\[ [A^{n+\rrr-2k}] \Delta_{\lambda_1} \cdots \Delta_{\lambda_k} F(\lambda_1,\dots,\lambda_k) = \\
k!\; \left. \poly_k^{\ttop}(\lambda_1+1,\dots,\lambda_k+1) \right|_{\gamma:=-1}.\]

Each application of a difference operator decreases the degree of a polynomial by one.
Together with assumption \ref{zero:topdegree} this implies that the left-hand side
is as a polynomial in $\lambda_1,\dots,\lambda_k$ 
of degree at most $n-\rrr-1$, so $p_k^{\ttop}\big|_{\gamma:=-1}$ must be also of degree at most $n-\rrr-1$.

Assumption \ref{zero:vanishing} implies that the left-hand side of 
\eqref{eq:difference-gives-something} is equal to zero for $|\lambda|\leq n-\rrr-1$; 
thus \cref{lem:lots-of-zeros} can be applied again to show that \eqref{eq:top-polynomial-top-model}
is the zero polynomial, as required.
\end{itemize}

\emph{This concludes the proof of the inductive step over the variable $k$.}
\end{proof}

\section{Approximate factorization property for the vertical arrow}
\label{sec:afp-vertical}

We recall that the conditional cumulants which correspond to the vertical arrow
in \eqref{eq:commutative-diagram-B2} between $\functions$ and $\functions_\row$
are denoted by $\kumuPointRow$.

\subsection{Closed formula for the cumulants $\kumuPointRow$}

Our goal in this section will be to find a closed formula
for the cumulant $\kumuPointRow(x_1,\dots,x_n)$ 
for $x_1,\dots,x_n\in\functions$.
By linearity of cumulants we may 
assume that for each value of the index $i$, the function $x_i\in \functions$ has the form
\begin{equation}
\label{eq:what-is-x}
 x_i(\lambda)= \sum_{j^{(i)}_1 < \dots  < j^{(i)}_{m(i)}} 
g_i\left(\lambda_{j^{(i)}_1},\dots,\lambda_{j^{(i)}_{m(i)}}\right).   
\end{equation}
It follows that the pointwise product of functions is given by
\begin{multline}
\label{eq:monster-product}
 ( x_1 \cdots x_n) (\lambda) = \\
\sum_{j^{(1)}_1 < \dots  < j^{(1)}_{m(1)}} \cdots \sum_{j^{(n)}_1 < \dots  < j^{(n)}_{m(n)}}  
\prod_{1\leq i\leq n} g_i\left(\lambda_{j^{(i)}_1},\dots,\lambda_{j^{(i)}_{m(i)}}\right).
\end{multline}

Let us fix some summand on the right-hand side. 
We denote
\begin{equation} 
\label{eq:sets-J}
J^{(i)}:=\{ j^{(i)}_1,  \dots  ,j^{(i)}_{m(i)}\}    
\end{equation}
and consider the graph $\mathcal{G}$
with the vertex set $[n]=\{1,2,\dots,n\}$ the elements of which correspond to the factors;
we draw an edge between the vertices $a$ and $b$ if the sets
$J^{(a)}$ and $J^{(b)}$
are not disjoint. The connected components of the graph $\mathcal{G}$
define a certain partition of the set $[n]$.

It follows that the right-hand side of \eqref{eq:monster-product}
can be written in the form
\begin{equation}
\label{eq:moment-cumulant-connected}
 x_1 \cdots x_n  = \sum_\nu  \prod_{b\in\nu} \widetilde{\kumuPointRow}(x_i : i\in b),    
\end{equation}
where the sum runs over all set-partitions $\nu$ of the set $[n]$ and
the product runs over the blocks of $\nu$. 
In the above formula $\widetilde{\kumuPointRow}$ denotes the contribution of a prescribed
connected component of the graph $\mathcal{G}$, i.e.
\newcommand{\rrrr}{l}
\begin{multline}
\label{eq:cumulants-point-row-explicit}
 \left( \widetilde{\kumuPointRow}(x_{i_1}, \dots, x_{i_\rrrr})\right)
(\lambda_1,\lambda_2,\dots) :=   \\
\sum_{j^{(i_1)}_1 < \dots  < j^{(i_1)}_{m(i_1)}} \cdots \sum_{j^{(i_\rrrr)}_1 < \dots  < j^{(i_\rrrr)}_{m(i_\rrrr)}} 
\prod_{1\leq k\leq \rrrr}
g_{i_k}\left(\lambda_{j^{(i_k)}_1},\dots,\lambda_{j^{(i_k)}_{m(i_k)}}\right)
\end{multline}
is defined as the sum over such choices of the indices that 
the restriction of the above graph $\mathcal{G}$ to the vertex set 
$\{i_1,\dots,i_\rrrr\}$ is a connected graph.

\medskip

It is a simple classical result (\emph{`the moment-cumulant formula'})
that the relation \eqref{eq:what-is-cumulant} between the cumulants
and moments can be inverted;
in our current setup this yields
\begin{equation}
\label{eq:moment-cumulant-connected-2}
 x_1 \cdots x_n  = \sum_\nu  \prod_{b\in\nu} {\kumuPointRow}(x_i : i\in b),
\end{equation}
where the sum runs over set-partitions of $[n]$.
Comparison of \eqref{eq:moment-cumulant-connected} with \eqref{eq:moment-cumulant-connected-2}
shows that the 
quantities $\widetilde{\kumuPointRow}$ fulfill the same recurrence relations
as the cumulants $\kumuPointRow$.
Since the system of equations \eqref{eq:moment-cumulant-connected-2} has the unique solution,
it follows that 
\[ {\kumuPointRow}=\widetilde{\kumuPointRow} \]
thus \eqref{eq:cumulants-point-row-explicit} gives an explicit formula for the latter cumulants.

In this way we proved the following result.
\begin{lemma}
\label{lem:cumulants-concretely}
If $x_i\in\functions$ are given by \eqref{eq:what-is-x}
then the corresponding cumulant
$\kumuPointRow(x_{i_1}, \dots, x_{i_r})$ is given by
the right-hand side of
\eqref{eq:cumulants-point-row-explicit}.   
\end{lemma}

In the following lemma we shall use the notations from the above proof.
Also, for a graph $\mathcal{G}$ we denote by $c(\mathcal{G})$  the number of its connected components.
\begin{lemma}
\label{lem:connected-components}
Let a family $J^{(1)},\dots,J^{(n)}$ of sets \eqref{eq:sets-J} be given.
\begin{enumerate}
   \item 
Assume that $\mathcal{G}'$ is a subgraph of $\mathcal{G}$ with the same vertex set $[n]$.
Then
\[ \sum_C  \left| \bigcup_{a\in C} J^{(a)}\right| \leq m(1)+\cdots+m(n)+ c(\mathcal{G}') -n, \]
where the first sum on the left-hand side runs over the connected components of $\mathcal{G}'$.

\item 
\[  
\left| \bigcup_{1\leq a\leq n} J^{(a)}  \right|
\leq  m(1)+\cdots+m(n)+ c(\mathcal{G})-n\]
\end{enumerate}
\end{lemma}
\begin{proof}
The proof of the first part of the lemma is a simple induction with respect to
the number of the edges of the graph $\mathcal{G}'$ based on the inclusion-exclusion principle
$|A\cup B|=|A|+|B|-|A\cap B|$.

The second part follows from the first part by setting $\mathcal{G}':=\mathcal{G}$.
\end{proof}

\subsection{The vertical arrow has approximate factorization property}

\begin{proposition}
\label{lem:vertical-arrow-R}
The vertical arrow from \eqref{eq:commutative-diagram-B2} has approximate factorization property.  
\end{proposition}
\begin{proof}
Let $x_1,\dots,x_n\in\functions$ be of the form \eqref{eq:what-is-x}.

\cref{lem:cumulants-concretely} 
gives explicitly the kernel $(f_r)$ for the cumulant
\[ \kumuPointRow(x_{1}, \dots, x_{n})
   = 
\sum_{r\geq 0} \sum_{i_1 < \dots < i_r} f_{r}(\lambda_{i_1},\dots,\lambda_{i_r}).   
\]
More specifically, the summand on the right-hand side
for some specified value of $r$ corresponds to the summands on the right-hand side of
\eqref{eq:cumulants-point-row-explicit} for which
\[ | J^{(1)} \cup \cdots \cup J^{(n)} | = r.\]

We keep notations from \eqref{eq:what-is-x}.
Assume that $x_i\in\functions$ is of degree at most $d_i$;
in other words we assume that the corresponding kernel
$g_i$ takes only values in Laurent polynomials of degree at most
$d_i - 2 m(i)$. It follows that the function $f_r$ takes values in Laurent
polynomials of degree at most
\[ \sum_{1\leq i\leq n} d_i - 2 m(i).\]

On the other hand, the second part of \cref{lem:connected-components} shows that non-zero
contribution can be obtained only for the values of $r$ which fulfill the bound
\[r= | J^{(1)} \cup \cdots \cup J^{(n)} |  \leq  m(1)+\cdots+m(n)+ 1-n.\]

It follows that  $\kumuPointRow(x_{1}, \dots, x_{n})$ is a row-function
of degree at most
\[ \left( \sum_{1\leq i\leq n} d_i - 2 m(i)\right) + 2 r \leq 
\left( \sum_{1\leq i\leq n} d_i \right) - 2 (n-1), \]
which concludes the proof.
\end{proof}

\section{Cumulants for the long diagonal arrow}
\label{sec:cumulants-long-diagonal}

Recall that we denote by $\kumuDisjointRow$ the cumulants which correspond to the long diagonal arrow
in \eqref{eq:commutative-diagram-B2} between $\Poly_\disjoint$ and $\functions_\row$.

\subsection{Vanishing on small Young diagrams}

\begin{lemma}
\label{property:cool-vanishing}
Assume that $a,b\geq 0$ are integers and 
$F,G\in\functions$ are row functions such that 
\begin{align*} 
F(\lambda)&=0 \qquad \text{holds for each $\lambda\in\AllYoung$ such that } |\lambda|<a, \\
\nonumber
 G(\lambda)&=0 \qquad \text{holds for each $\lambda\in\AllYoung$ such that } |\lambda|<b.
\intertext{\indent Then}
\nonumber
 (F\row G)(\lambda)&=0 \qquad \text{holds for each $\lambda\in\AllYoung$ such that } |\lambda|<a+b. 
\end{align*}
\end{lemma}
\begin{proof}
Let $(f_r)$ be the kernel of $F$, see \eqref{eq:row-functions}.
The right-hand side of \eqref{eq:row-functions} 
involves only the values of $f_{r}$ over $r\leq \ell(\lambda)$
thus the collection of equalities \eqref{eq:row-functions} 
can be viewed as an upper-triangular system of linear equations. 
It follows immediately that 
\[ f_r(x_1,\dots,x_r) =0 \]
holds true for all $r\geq 0$ and all non-negative integers $x_1,\dots,x_r$
such that
\[ x_1+\cdots+x_r<a.\]

An analogous property is fulfilled by the kernel of the row function
$G$ (with the variable $a$ replaced by $b$).

From the very definition of the disjoint product it follows that also the kernel
of $F\otimes G$ also fulfills this property (with the variable $a$ replaced by $a+b$).
\end{proof}

\subsection{M\"obius invertion}

It is easy to show from the very definition \eqref{eq:what-is-cumulant} that
a cumulant 
\[ \kappa(X_1,\dots,X_n) \]
is a linear combination (with rational coefficients)
of the expressions of the form
\begin{equation}
\label{eq:sample-product}
 \prod_{b\in \nu}  \EE\left(  \prod_{i\in b} X_i \right)    
\end{equation}
over set-partitions of the set $[n]$.

We will show now that if $n\geq 2$ then the sum of these coefficients is equal to zero.
Indeed, if we set $X_1=\cdots=X_n=1$ to be the unit of the algebra
then from the very definition \eqref{eq:what-is-cumulant} 
of the cumulants it follows that the corresponding cumulant
$\kappa(1,\dots,1)=0$ vanishes while each product \eqref{eq:sample-product} is equal to $1$.

By specifying the conditional expectation $\EE=\id\colon\Poly_\disjoint\to\functions$ 
to be the long diagonal arrow in \eqref{eq:commutative-diagram-B2},
we have proved the following result.

\begin{lemma}[M\"obius invertion]
\label{lem:cumulants-expression}
For any partitions $\pi_1,\dots,\pi_l$ the function 
\begin{equation}
\label{eq:cond-kumu}
\kumuDisjointRow(\Ch_{\pi_1},\dots,\Ch_{\pi_l})   
\end{equation}
is a linear combination (with rational coefficients) of expressions of the form
\begin{equation*}
\bigotimes_{b\in \nu} \Ch_{\prod_{i\in b} \pi_i}   
\end{equation*}
over set-partitions $\nu$ of the set $[l]$. For example, in the case $l=3$
the function \eqref{eq:cond-kumu} is a linear combination of the following
five expressions:
\begin{multline*}
\Ch_{\pi_1}\otimes \Ch_{\pi_2} \otimes \Ch_{\pi_3},  \quad 
\Ch_{\pi_1 \pi_2} \otimes \Ch_{\pi_3},  \quad 
\Ch_{\pi_1 \pi_3} \otimes \Ch_{\pi_2},  \\
\Ch_{\pi_2 \pi_3} \otimes \Ch_{\pi_1},  \quad 
\Ch_{\pi_1 \pi_2 \pi_3}.    
\end{multline*}

Furthermore, if $l\geq 2$ then the sum of the coefficients in this linear combination is equal to $0$.   

\medskip

The above results hold true also for the cumulants $\kumuDisjointPoint$; in the latter case
the product $\otimes$ should be replaced by the usual pointwise multiplication of functions
on the set $\AllYoung$ of Young diagrams.
\end{lemma}

\subsection{Conditional cumulants for the diagonal arrow}

\begin{proposition}
\label{lem:disjoint-row-vanish-nice}
For any partitions $\pi_1,\dots,\pi_n$
\[  \kumuDisjointRow(\Ch_{\pi_1},\dots,\Ch_{\pi_n})(\lambda)=0 \]
holds true for any Young diagram $\lambda$ such that
$|\lambda|< |\pi_1|+\dots+|\pi_n|$.
\end{proposition}
\begin{proof}
This result is a straightforward consequence of 
\cref{lem:cumulants-expression} and \cref{property:cool-vanishing}.
\end{proof}

\section{Proof of the second main result: approximate factorization property}
\label{sec:key-tool-proof}

\subsection{Conditional cumulants for a commutative diagram}

\begin{lemma}[\cite{Brillinger}] 
\label{lem:iterated-cumulants}
Assume that $\Alg$, $\AlgB$ and $\AlgC$ are commutative unital algebras 
and let $\condExp{\Alg}{\AlgB}$, $\condExp{\AlgB}{\AlgC}$ and $\condExp{\Alg}{\AlgC}$
be unital maps between them such that the following diagram commutes:
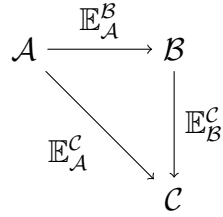
\[ \begin{tikzpicture}[node distance=2cm, auto, baseline=-5ex]
  \node (A) {$\Alg$};
  \node (B) [right of=A] {$\AlgB$};
  \node (C) [below of=B]{$\AlgC$};

  \draw[->] (A) to node {$\condExp{\Alg}{\AlgB}$} (B);
  \draw[->] (B) to node {$\condExp{\AlgB}{\AlgC}$} (C);
  \draw[->] (A) to node [swap]{$\condExp{\Alg}{\AlgC}$} (C);
\end{tikzpicture}
\]

Then
\[  \condKumu{\Alg}{\AlgC}(x_1,\dots,x_n) = \sum_{\nu\in\partitions(n)} 
\condKumu{\AlgB}{\AlgC} \Big( \condKumu{\Alg}{\AlgB}(x_i : i\in b)      : b \in \nu \Big).
\]  
\end{lemma}
\begin{example}
\begin{align*}
    \condKumu{\Alg}{\AlgC}(x_1)         &= \condKumu{\AlgB}{\AlgC} \Big( \condKumu{\Alg}{\AlgB}(x_1) \Big), \\ 
    \condKumu{\Alg}{\AlgC}(x_1,x_2)     &= \condKumu{\AlgB}{\AlgC} \Big( \condKumu{\Alg}{\AlgB}(x_1,x_2) \Big)+
\condKumu{\AlgB}{\AlgC} \Big( \condKumu{\Alg}{\AlgB}(x_1), \condKumu{\Alg}{\AlgB}(x_2) \Big).
\end{align*}

\end{example}

\subsection{Proof of the second main result}
We are now ready to show the proof of \cref{theo:factorization-of-characters}.
For Reader's convenience we will restate this theorem in the following form.

\begin{theorem}[Reformulation of \cref{theo:factorization-of-characters}]
For any partitions $\pi_1,\dots,\pi_l$ the conditional cumulant
\[ F:=\kumuDisjointPoint(\Ch_{\pi_1},\ldots,\Ch_{\pi_l})\in\Poly \]
is of degree at most 
\[  |\pi_1|+\cdots+|\pi_l|+\ell(\pi_1)+\cdots+\ell(\pi_l)-2(l-1).\]
\end{theorem}
\begin{proof}
We use induction over $l$. For the induction base $l=1$
\[ F=\kumuDisjointPoint(\Ch_{\pi_1})=\Ch_{\pi_1} \]
and there is nothing to prove.
In the following we shall
consider the case $l\geq 2$; we assume that the statement of the theorem holds true for all $l'<l$.

\bigskip

\newcommand{\variabledelta}{j}
\newcommand{\Bigvariabledelta}{J}

We start with an observation that if for some value of the index $i$ we have
$\pi_i=\emptyset$ then $\Ch_{\pi_i}=1$ is the unit in $\Poly$ thus
(for $l\geq 2$) the corresponding cumulant vanishes
by the very definition \eqref{eq:what-is-cumulant}:
\[ F=\kumuDisjointPoint(\Ch_{\pi_1},\dots,1,\dots,\Ch_{\pi_l})=0 \]
and the claim holds true trivially. From the following on we shall assume that
$\pi_1,\dots,\pi_l\neq \emptyset$ are all non-empty.

\medskip

We denote
\[ d:= |\pi_1|+\cdots+|\pi_l|+\ell(\pi_1)+\cdots+\ell(\pi_l).\]
We will use a nested induction over $\Bigvariabledelta\in\{0,\dots,2l-2\}$ and show that
$F$ is of degree at most $d-\Bigvariabledelta$.
This result (for the special choice $\Bigvariabledelta=2l-2$) 
would finish the proof of the inductive step with respect 
to the variable $l$ and thus would conclude the proof.

\medskip

We start by noticing that
M\"obius invertion (\cref{lem:cumulants-expression} in the alternative formulation,
for the cumulants $\kumuDisjointPoint$)
implies that $F\in\Poly$ is of degree at most
$d$ and thus the induction base $\Bigvariabledelta=0$ holds trivially true.

\bigskip
\emph{The inductive step over $\Bigvariabledelta$.}
The inductive hypothesis with respect to the variable $\Bigvariabledelta$
states that $F$ is of degree (at most) $d-\Bigvariabledelta$
for some choice of $\Bigvariabledelta\in\{0,\dots,2l-3\}$.
We define $j\in\{0,\dots,l-2\}$ by setting
\begin{align*}
   J &= \begin{cases} 
              2j  &\text{if $J$ is even},\\
              2j+1  &\text{if $J$ is odd}
         \end{cases}
\intertext{and set}
n & := |\pi_1|+\cdots+|\pi_l|-\variabledelta \geq 2,\\
\rrr &:= \ell(\pi_1)+\cdots+\ell(\pi_l)+\variabledelta-\Bigvariabledelta \geq 1.
\intertext{In this way}
n+\rrr &= d- J, \\
n-\rrr &\geq |\pi_1|+\cdots+|\pi_l| - \ell(\pi_1)- \cdots- \ell(\pi_l).
\end{align*}

Our strategy is to apply \cref{lem:keylemma} 
either: 
\begin{itemize}
   \item in the original formulation (in the case $\variabledelta=0$), or,
   \item in the alternative formulation (in the case $\variabledelta\geq 1$)
\end{itemize}
for the above choice of $n$ and $\rrr$.
We first check that its assumptions are fulfilled.

\medskip

\emph{Assumption \ref{zero:initial-bound}.} 
This assumption is just the inductive hypothesis.

\medskip

\emph{Assumption \ref{zero:topdegree}.}
We have to verify this assumption only in the case $\variabledelta=0$.
By M\"obius invertion (\cref{lem:cumulants-expression}) and 
\cref{prop:vershik-kerov-jack-character}
it follows that
\begin{equation}
\label{eq:VK-applied}
\AllYoung \ni(\lambda_1,\dots,\lambda_m) \mapsto F(\lambda_1,\dots,\lambda_m)\in\Laurent      
\end{equation}
is a priori a polynomial of degree $|\pi_1|+\cdots+|\pi_l|$ and its homogeneous top-degree part is equal to
some multiple of
\[ A^{|\pi_1|+\cdots+|\pi_l|-\ell(\pi_1)-\cdots-\ell(\pi_l)}\ 
p_{\pi_1\cdots\pi_l}(\lambda_1,\dots,\lambda_m).  \]
However, since $l\geq 2$, the second part of \cref{lem:cumulants-expression}
implies that this multiple is actually equal to zero. 
In other words, \eqref{eq:VK-applied} is a polynomial of degree at most 
$|\pi_1|+\cdots+|\pi_l|-1=n-1$, as required.

\medskip

\emph{Assumption \ref{zero:laurent}.} 
M\"obius invertion (\cref{lem:cumulants-expression}) and 
\cref{prop:degree-laurent} 
imply that for any Young diagram $\lambda$
the evaluation $F(\lambda)$ is a Laurent polynomial of degree at most 
\[|\pi_1|+\cdots+|\pi_l|-\ell(\pi_1)-\cdots-\ell(\pi_l) < n-r+1\]
as required.

\medskip

\emph{Assumptions \ref{zero:vanishing} and \ref{zero:vanishing-B}.}
Our strategy is to consider the following simplified version of the
commutative diagram \eqref{eq:commutative-diagram-B2}.
\begin{equation}
\label{eq:commutative-diagram-B2-simple}
\begin{tikzpicture}[node distance=2cm, auto, baseline=-5ex]
  \node (A) {$\Poly_\disjoint$};
  \node (B) [right of=A] {$\Poly$};
  \node (C) [below of=B]{$\functions_\row$};

  \draw[->] (A) to node {$\id$} (B);
  \draw[->] (B) to node {$\id$} (C);
  \draw[->] (A) to node [swap]{$\id$} (C);
\end{tikzpicture}
\end{equation}
Recall that the conditional cumulants which correspond to 
the horizontal arrow are denoted by $\kumuDisjointPoint$;
the ones which correspond to the vertical arrow are denoted by $\kumuPointRow$;
and the ones which correspond to the diagonal arrow are denoted by $\kumuDisjointRow$.

Since $\kumuPointRow(x)=x$ it follows that
\begin{multline}
\label{eq:diagonal-vertical-horizontal}
F=\kumuDisjointPoint(\Ch_{\pi_1},\dots,\Ch_{\pi_l})=
\kumuPointRow \big( \kumuDisjointPoint(\Ch_{\pi_1},\dots,\Ch_{\pi_l})  \big) = \\ 
\kumuDisjointRow(\Ch_{\pi_1},\dots,\Ch_{\pi_l}) - 
\sum_{
\nu\neq \mathbf{1}}
\kumuPointRow \Big( \kumuDisjointPoint(\Ch_{\pi_i} : i\in b)   : b \in \nu \Big),
\end{multline}
where the last equality follows from Brillinger's formula (\cref{lem:iterated-cumulants});
the sum on the right-hand side runs over set-partitions of $[l]$
which are different from the maximal partition $\mathbf{1}=\big\{ \{1,\dots,l\} \big\}$.
We will substitute the summands which contribute to the right-hand side into \eqref{eq:top-key-equation} 
and we will investigate the resulting expressions.

\medskip

Firstly, \cref{lem:disjoint-row-vanish-nice} shows that 
$\big( \kumuDisjointRow(\Ch_{\pi_1},\dots,\Ch_{\pi_l})\big)(\lambda)=0$ 
for all $\lambda\in\AllYoung$ such that $|\lambda|\leq  |\pi_1|+\dots+|\pi_l|-1$ thus
for any integer $k\geq r$
\begin{equation}
\label{eq:zerozero}
 \Delta_{\lambda_1} \cdots \Delta_{\lambda_k} 
\big( \kumuDisjointRow(\Ch_{\pi_1},\dots,\Ch_{\pi_l})\big)^{\sym} (\lambda) =0
\end{equation}
for all $\lambda\in\AllYoung$ such that 
\[|\lambda| \leq |\pi_1|+\dots+|\pi_l|-1-k = (n+\rrr-2k) + (\variabledelta-1)+\underbrace{( k-\rrr)}_{\geq 0}.\]

\medskip

Secondly, let us fix the value of the set-partition $\nu\neq \mathbf{1}$ and 
let us investigate the corresponding summand on the right-hand side of \eqref{eq:diagonal-vertical-horizontal}.
From the inductive hypothesis over the variable $l$ it follows for each block $b\in\nu$ that
the cumulant $\kumuDisjointPoint(\Ch_{\pi_i} : i\in b)\in\Poly$ 
is of degree at most 
\[ \left(\sum_{i\in b} |\pi_i|+\ell(\pi_i) \right)-2( |b|-1).\]
Thus the approximate factorization property for the vertical arrow
(\cref{lem:vertical-arrow-R}) implies that
\begin{equation*}
G:= \kumuPointRow \Big( \kumuDisjointPoint(\Ch_{m_i} : i\in b) : b \in \nu \Big) 
\in\functions   
\end{equation*}
is a row function of degree (at most)
\begin{multline*}
 \sum_{b\in\nu } \left[ 
\left(\sum_{i\in b} |\pi_i|+\ell(\pi_i) \right)-2( |b|-1)
 \right] +2 -2 |\nu|  =\\ 
d-2(l-1) \leq (d-J)-1 =(n+r)-1.
\end{multline*}
The latter bound on the degree of $G$ and \cref{coro:small-degree-killed} imply that
\begin{equation}
\label{eq:zerozero2}
 [A^{n+\rrr-2k}] \Delta_{\lambda_1} \dots \Delta_{\lambda_k}  
 G(\lambda_1,\dots,\lambda_k)=0
\end{equation}
for an arbitrary choice of the integers $\lambda_1,\dots,\lambda_k\geq 0$.

\medskip

From \eqref{eq:zerozero} and \eqref{eq:zerozero2} it follows that:
\begin{itemize}
   \item if $\variabledelta=0$     then condition \ref{zero:vanishing} holds true;
   \item if $\variabledelta\geq 1$ then stronger condition \ref{zero:vanishing-B} holds true.
\end{itemize}

\smallskip

\emph{Conclusion.}
In this way we verified that the assumptions of \cref{lem:keylemma} are
indeed fulfilled. It follows therefore that
$F$ is of degree at most $d-\Bigvariabledelta-1$.
This concludes the proof of the induction step over the variable $\Bigvariabledelta$.
\end{proof}

\section{$\Chtt_n\in\Poly$ is of degree $n+1$}
\label{sec:degree-n+1}

The current section is devoted to the proof of the following
result which will be essential for the proof of \cref{theo:second-main-bis}.
\begin{proposition}
\label{prop:candidate-is-polynomial}
For each integer $n\geq 1$ the function 
$\Chtt_n\in\Poly$ (cf.~\eqref{eq:top-top-top})  
is an $\alpha$-polynomial function of degree at most $n+1$.
\end{proposition}

The proof is split into two parts:
first in \cref{prop:candidate-is-polynomial} we will show that $\Chtt_n\in\Poly$
and then in \cref{sec:degree-bound-chttn}  we will show the degree bound.

\subsection[]{The first part of proof of \cref{prop:candidate-is-polynomial}: $\Chtt_n\in\Poly$}
\label{sec:is-a-polynomial}

\newcommand{\cumulantB}[2]{\mathcal{K}_{#1,#2}}
\newcommand{\cumulant}[1]{\mathcal{K}_{#1}}
\newcommand{\moment}[1]{\mathcal{M}_{#1}}

Our proof of the claim that $\Chtt_n\in\Poly$ 
will be based on the following result
(the proof of which is postponed to \cref{sec:missing-proof-of-polynomial}).
\begin{proposition}
\label{prop:something-simpler-is-polynomial}
For any integer $n\geq 1$ and permutation $\pi\in\Sym{n}$ 
the function
on $\AllYoung$ given by
\begin{equation}
\label{eq:cumulant-definition}
 \cumulant{\pi} := 
(-1)^{|C(\pi)|}
\sum_{\substack{
\sigma_1,\sigma_2\in \Sym{n}, \\ \sigma_1 \sigma_2= \pi, \\ 
\langle \sigma_1,\sigma_2 \rangle \text{ is transitive}
}}  
 \Embed_{\sigma_1,\sigma_2}
\end{equation}
is an $\alpha$-polynomial function.
\end{proposition}

For an integer $l$ we consider
the homogeneous part of degree $l$ of the
anisotropic Stanley polynomial for $\cumulant{\pi}$.
This new anisotropic Stanley polynomial defines a function on $\AllYoung$
which is explicitly given by
\begin{equation}
\label{eq:polypolypoly}
 \cumulant{\pi}^l :=
 (-1)^{|C(\pi)|}
\sum_{\substack{
\sigma_1,\sigma_2\in \Sym{n}, \\ \sigma_1 \sigma_2= \pi, \\ 
|C(\sigma_1)|+|C(\sigma_2)|=l,\\
\langle \sigma_1,\sigma_2 \rangle \text{ is transitive}
}}  
 \Embed_{\sigma_1,\sigma_2}.
\end{equation}

On the other hand, by \cref{prop:generate-the-same} the function $\cumulant{\pi}\in\Poly$ can be expressed as a
polynomial in the indeterminates $\gamma,\Sfunct_2,\Sfunct_3,\dots$.
To these indeterminates we associate the degrees as in \eqref{eq:filtration2};
the corresponding 
homogeneous part of degree $l$ of this polynomial 
is clearly an element of $\Poly$.
By \cref{eq:S-is-for-Stanley} this polynomial is equal to $\cumulant{\pi}^l$.
In this way we proved that $\cumulant{\pi}^l\in\Poly$.

Since $\Chtt_n$ is a linear combination of such functions,
this completes the proof of the claim that $\Chtt_n\in\Poly$.

\subsection{Proof of \cref{prop:something-simpler-is-polynomial}}
\label{sec:missing-proof-of-polynomial}
\begin{proof}[Proof of \cref{prop:something-simpler-is-polynomial}]
If $X$ is an arbitrary set, we denote by $\partitions(X)$ the set of all set partitions of $X$.
Let $\pi\in\Sym{n}$ be a fixed permutation.
We denote
\[ \partitions_\pi:=\{ P \in \partitions([n]) : P \geq C(\pi)\}\]
which is the set of the partitions $P$ of the underlying set $[n]$ which have the property 
that each cycle $c\in C(\pi)$ 
is contained in one of the blocks of the partition $P$.
We define
\begin{equation}
\label{eq:moment-cumulant}
 \moment{\pi}(\lambda) := \sum_{P\in\partitions_\pi}  \prod_{B\in P} \cumulant{\pi|_{B}}(\lambda),     
\end{equation}
where 
the product runs over the blocks of the partition $P$ and
$\pi|_{B}: B \to B$ denotes the restriction of the permutation $\pi$ to the set $B\subseteq [n]$.

It is not hard to see that
\begin{equation}
\label{eq:moment-cumulant-nice}
 \moment{\pi}(\lambda) =
(-1)^{|C(\pi)|}
\sum_{\substack{\sigma_1,\sigma_2\in \Sym{n}, \\ \sigma_1 \sigma_2= \pi}}  
 \Embed_{\sigma_1,\sigma_2}(\lambda)
\end{equation}
(the difference between the right-hand side of \eqref{eq:moment-cumulant-nice} and \eqref{eq:cumulant-definition}
lies in the requirement on transitivity);
indeed, the summands on the right-hand side of \eqref{eq:moment-cumulant-nice} can be pigeonholed according 
to the set of orbits of the group $\langle \sigma_1,\sigma_2\rangle$ and each such a class of summands corresponds to
an appropriate summand on the right-hand side of \eqref{eq:moment-cumulant}.

\smallskip

The function $\cumulant{\pi}$ (respectively, the function $\moment{\pi}$)
depends only on the conjugacy class of the permutation $\pi$.
Since such conjugacy classes are in a bijective correspondence with the integer partitions,
we may index the family $(\cumulant{\pi})$ (respectively, the family $(\moment{\pi})$)
by $\pi$ being an integer partition.
With this perspective, \eqref{eq:moment-cumulant} can be viewed as the following system of 
equalities:
\begin{equation}
\label{eq:moment-cumulant-2}
\left\{ 
\begin{aligned}
  \moment{i} &=  \cumulant{i} & \text{for all $i\geq 1$}, \\
  \moment{i,j} &=  \cumulant{i,j}+ \cumulant{i} \cumulant{j}
& \text{for all $i\geq j\geq 1$}, \\ 
  \moment{i,j,k} &=  \cumulant{i,j,k}+ \cumulant{i} \cumulant{j,k}
+ \cumulant{j} \cumulant{i,k}
+ \cumulant{j} \cumulant{i,k} + \hspace{-7ex} 
\\ &
+ \cumulant{i} \cumulant{j} \cumulant{k}
& \text{for all $i\geq j\geq k \geq 1$},\\
& \vdots 
\end{aligned}
\right.
\end{equation}
If we view $(\cumulant{\pi})$ as variables, \eqref{eq:moment-cumulant-2} 
becomes an upper-triangular system of algebraic equations which can be solved.
This shows that each $\cumulant{\sigma}$ can be expressed as a polynomial in the variables
$(\moment{\pi})$. 
Thus for our purposes it is enough to show that
$\moment{\pi}$ is an $\alpha$-polynomial function.
This is exactly the content of \cref{lem:non-transitive-polynomial} which we prove below.
\end{proof}

\begin{lemma}
\label{lem:non-transitive-polynomial}
For each integer $n\geq 1$ and permutation $\pi\in\Sym{n}$
the function  
\[
\lambda\mapsto  \sum_{\substack{
\sigma_1,\sigma_2\in \Sym{n} \\ 
\sigma_1 \sigma_2=\pi \\
}}  
\ \Embed_{\sigma_1,\sigma_2}(\lambda).
\]
is an $\alpha$-polynomial function.
\end{lemma}
\begin{proof}
The specialization $A=1$ of the Jack characters coincides with the usual characters
of the symmetric groups 
$\lambda\mapsto \Ch^{\alpha=1}_\pi(\lambda)$
for which several convenient results are known. 

Firstly, the following closed formula for the characters of the symmetric groups is known,
cf.~\cite[Theorem 2]{FeraySniady2011a}:
\begin{multline*}
\Ch^{\alpha=1}_\pi  =
\sum_{\substack{\sigma_1,\sigma_2\in \Sym{n}, \\ \sigma_1 \sigma_2= \pi}}  
(-1)^{\sigma_1} N_{\sigma_1,\sigma_2} = \\
(-1)^{|C(\pi)|} \sum_{\substack{\sigma_1,\sigma_2\in \Sym{n}, \\ \sigma_1 \sigma_2= \pi}}  
(-1)^{|C(\sigma_2)|} N_{\sigma_1,\sigma_2}
\end{multline*} 
where both sides of the above equality should be viewed as functions on $\AllYoung$.

Secondly, there exists a multivariate polynomial $H\in\Q[s_2,\dots,s_{2n}]$ with the property that
\[ \Ch^{\alpha=1}_\pi = H(S_2,S_3,\dots,S_{2n}), \]
where $S_2,\dots,S_{2n}$ are the isotropic functionals of shape \eqref{eq:what-is-isotropic-s};
the existence and the explicit form of this polynomial $H$ were discussed in
\cite{DolegaFeraySniady2008}.

By combining the above two results it follows that
\begin{align}
\label{eq:alpha-1-magic-sum}
 \sum_{\substack{\sigma_1,\sigma_2\in \Sym{n}, \\ \sigma_1 \sigma_2= \pi}}  
(-1)^{|C(\sigma_2)|} N_{\sigma_1,\sigma_2} 
&=  (-1)^{|C(\pi)|}   H(S_2,S_3,\dots,S_{2n}),    
\intertext{where both sides should be viewed as functions on $\AllYoung$.
We claim that}
\label{eq:flaming-lips}
 \sum_{\substack{\sigma_1,\sigma_2\in \Sym{n}, \\ \sigma_1 \sigma_2= \pi 
} } 
\Embed_{\sigma_1,\sigma_2} 
&=  (-1)^{|C(\pi)|}   H(\Sfunct_2,\Sfunct_3,\dots,\Sfunct_{2n});     
\end{align}
indeed \cref{lem:stanley-polunomial-isotropic-and-anisotropic-the-same} and \cref{eq:S-is-for-Stanley}
imply that isotropic Stanley polynomials for both sides of 
\eqref{eq:alpha-1-magic-sum} are equal to the corresponding 
anisotropic Stanley polynomials of both sides of \eqref{eq:flaming-lips}.

An application of \cref{prop:generate-the-same} implies that \eqref{eq:flaming-lips}
is an $\alpha$-polynomial function, as required.
\end{proof}

\subsection{The second part of proof of \cref{prop:candidate-is-polynomial}: degree bound on $\Chtt_n\in\Poly$}
\label{sec:degree-bound-chttn}

We are ready to prove the remaining part of \cref{prop:candidate-is-polynomial}, namely the degree bound.
We already proved in \cref{sec:is-a-polynomial} that $\Chtt_n\in\Poly$;
it follows therefore by \cref{prop:generate-the-same} that there exists a 
multivariate polynomial $H(\gamma,s_2,s_3,\dots)$ with the property that
\begin{equation}
\label{eq:ch-in-s}
\Chtt_n = H(\gamma,\Sfunct_2,\Sfunct_3,\dots),   
\end{equation}
where we use the usual substitution \eqref{eq:gamma} for the variable $\gamma$.
We associate to the indeterminates $\gamma,s_2,s_3,\dots$ the degrees by setting
\begin{equation}
\label{eq:gradacja-ZZtop}
  \left\{ \begin{aligned} \degg \gamma &= 1, \\ 
             \degg s_n &= n, \qquad \text{for }n\geq 2.
          \end{aligned} \right. 
\end{equation}
\emph{We will prove below that with respect to this choice of degrees, $H$
is a homogeneous polynomial of degree $n+1$.} 

\medskip

Indeed; let $H_d(\gamma,s_2,s_3,\dots)$ denote the homogeneous part of
the polynomial $H$ of degree $d$. 
It follows by \cref{eq:S-is-for-Stanley} that the anisotropic Stanley polynomial
which corresponds to the specialization
\begin{equation}
\label{eq:yes-you-are-special}
H_d(\gamma,\Sfunct_2,\Sfunct_3,\dots)   
\end{equation}
is homogeneous of degree $d$;
thus the anisotropic Stanley polynomial for \eqref{eq:yes-you-are-special} 
is the homogeneous part of degree $d$
of the anisotropic Stanley polynomial for \eqref{eq:ch-in-s}.

On the other hand, 
from the way $\Chtt_n$ was defined in \eqref{eq:top-top-top} 
and by \cref{lem:stanley-polunomial-isotropic-and-anisotropic-the-same}
it follows that the anisotropic Stanley polynomial for $\Chtt_n$ is homogeneous
of degree $n+1$.

By combining the above two observations it follows that if $d\neq n+1$ 
then Stanley polynomial for \eqref{eq:yes-you-are-special} must be the zero polynomial.
Finally, it follows that $H_d=0$ for $d\neq n+1$. 
This completes the proof that $H$ is homogeneous of degree $n+1$.

\medskip

The degree bound on the polynomial $H$ implies, by \cref{prop:generate-the-same}, the desired
bound on the degree of $\Chtt_n$, regarded as an element of the filtered algebra $\Poly$.
This concludes the proof of \cref{prop:candidate-is-polynomial}.

\section{Proof of the first main result: top-degree of Jack characters}
\label{sec:proof}

\begin{proof}[Proof of  \cref{theo:second-main-bis}]
We define functions $F$, $F_1$, and $F_2$ by
\[ F:= \underbrace{\Ch_n}_{F_1:=} - \underbrace{\Chtt_n}_{F_2:=}; \]
we will show that $F$ fulfills the assumptions of \cref{lem:keylemma} for $r=1$.

\bigskip

\emph{Condition \ref{zero:initial-bound}.}
An analogue of condition \ref{zero:initial-bound} is fulfilled both for $F_1$
(cf.~\cref{thm:degree-of-Jack-character})
as well as for $F_2$
(cf.~\cref{prop:candidate-is-polynomial}).
It follows that \ref{zero:initial-bound} is satisfied for $F=F_1-F_2$, as required.

\bigskip

\emph{Condition \ref{zero:topdegree}.}
The polynomial
\begin{equation}
\label{eq:VK-chtt}
 \AllYoung\ni(\lambda_1,\dots,\lambda_m) \mapsto \Chtt_n (\lambda_1,\dots,\lambda_m)\in\Laurent     
\end{equation}
is equal to the anisotropic Stanley polynomial for $\Chtt_n$ evaluated for the anisotropic coordinates
\[ P=\left( A^{-1},  \dots,A^{-1} \right), \qquad 
   Q=\left( A \lambda_1,\dots, A\lambda_m\right).
\]
Thus, by \cref{lem:stanley-polunomial-isotropic-and-anisotropic-the-same},
the top-degree part of the polynomial \eqref{eq:VK-chtt} corresponds to the summands
$(\sigma_1,\sigma_2)$ from the definition \eqref{eq:top-top-top} of $\Chtt_n$
for which the number of cycles of $\sigma_2$ takes the maximal possible value;
in other words $\sigma_1=\id$ must be the identity permutation.
The transitivity requirement implies that in this case $\sigma_2$
must have exactly one cycle; there are $(n-1)!$ such permutations.
It follows that the polynomial \eqref{eq:VK-chtt} is of degree $n$ and its homogeneous
top-degree part is equal to 
\[A^{n-1}\ p_n(\lambda_1,\dots,\lambda_m).\]

On the other hand, \cref{prop:vershik-kerov-jack-character} gives the same polynomial
as the leading term for the map
\[ \AllYoung\ni(\lambda_1,\dots,\lambda_m) \mapsto \Ch_n (\lambda_1,\dots,\lambda_m)\in\Laurent.\]  

\smallskip

Due to the cancellation, it follows that
\[ \AllYoung\ni(\lambda_1,\dots,\lambda_m) \mapsto F(\lambda_1,\dots,\lambda_m) \] 
is a polynomial of degree (at most) $n-1$.
This completes the proof of condition \ref{zero:topdegree}.

\bigskip

\newcommand{\newell}{k}
\newcommand{\FunctionG}{H}

\emph{Condition \ref{zero:vanishing}.}
Since 
\begin{equation} 
\label{eq:small-boxes-vanishes}
F_1(\lambda)=\Ch_n(\lambda)=0 \qquad \text{if $|\lambda|<n$},   
\end{equation}
it follows that
\begin{equation} 
\label{eq:small-boxes-vanishes-B}
 \Delta_{\lambda_1} \cdots \Delta_{\lambda_k} F_1^{\sym}(\lambda_1,\lambda_2,\dots) = 0 \qquad 
\text{if $|\lambda|+k< n$};
\end{equation}
it follows that an analogue of condition \ref{zero:vanishing} is indeed fulfilled for $F_1$.

\bigskip

In the remaining part of the proof we concentrate on $F_2$.
The definition of an embedding, as well as the definitions of $N_{G}(\lambda)$ and $\Embed_{G}(\lambda)$ 
from \cref{sec:number-of-embeddings} 
can be naturally extended to 
an arbitrary tuple $\lambda=(\lambda_1,\dots,\lambda_\newell)$ of non-negative integers
(which does not necessarily forms a Young diagram);
condition \eqref{embedding:young} should be simply replaced by
\[
 1\leq f_1(w) \leq \lambda_{f_2(b)}
\]
for each pair of vertices $w\in V_{\circ}$, $b\in V_{\bullet}$ connected by an edge.
It is easy to check that so defined $N_G(\lambda_1,\dots,\lambda_\newell)$ is a symmetric function of its $\newell$
arguments.
Thus $\Chtt_n(\lambda_1,\dots,\lambda_\newell)$ given by \eqref{eq:top-top-top} is a symmetric
function of its $\newell$ arguments; in other words $F_2^{\sym}=F_2=\Chtt_n$
and no additional symmetrization is necessary.

For $\sigma_1,\sigma_2\in\Sym{n}$, any embedding $(f_1,f_2)$ of the bicolored graph $G_{\sigma_1,\sigma_2}$ 
into $\lambda=(\lambda_1,\dots,\lambda_\newell)$
can be alternatively viewed as a pair of functions
\[ f_1:[n]\rightarrow\N, \qquad f_2:[n]\rightarrow [\newell] \]
with the property that $f_s$ is constant on each cycle of $\sigma_s$ for $s\in\{1,2\}$ and such that
\begin{equation}
\label{eq:embedding-is-nice}
 1\leq f_1(m) \leq \lambda_{f_2(m)} \qquad \text{holds true for any $m\in[n]$}.   
\end{equation}
It follows that the sum on the right-hand-side of
\eqref{eq:top-top-top} can be alternatively written as
\begin{multline}
\label{eq:curly-bracket}
\Chtt_n(\lambda_1,\dots,\lambda_\newell)=\\
\ \hspace{-30ex}
{\sum_{\substack{f_1\colon[n]\rightarrow\N, \\ f_2\colon[n]\rightarrow[\newell], \\
\text{condition \eqref{eq:embedding-is-nice} holds true}}}
\sum_{\substack{ \sigma_2\in\Sym{n}, \\ \text{$f_2$ is constant on each cycle of $\sigma_2$}}} 
}
\\ 
\underbrace{ \Bigg\{ \sum_{\substack{ \sigma_1\in\Sym{n}, \\ \text{$f_1$ is constant on each cycle of $\sigma_1$,} \\ 
\text{$\langle \sigma_1,\sigma_2\rangle$ is transitive}}}\!\!\!\!\!\!\!\!\!\!
\gamma^{n+1-|C(\sigma_1)|-|C(\sigma_2)|}\;
A^{|C(\sigma_1)|} \left(\frac{-1}{A}\right)^{|C(\sigma_2)|}  \Bigg\}
}_{\FunctionG(\lambda_1,\dots,\lambda_\newell):=}.
\end{multline}

\medskip

Let us fix the values of $f_1$, $f_2$ and $\sigma_2$; we denote by $\FunctionG(\lambda_1,\dots,\lambda_\newell)$ 
the value of the curly bracket in the above expression \eqref{eq:curly-bracket}. 
We will investigate in the following 
the contribution of $\FunctionG$ to 
\begin{equation}
\label{eq:here-contribution}
 [A^{n+r-2k}] \Delta_{\lambda_1} \cdots \Delta_{\lambda_k} \Chtt_n(\lambda_1,\dots,\lambda_k),  
\end{equation}
cf.~the left-hand-side of \eqref{eq:top-key-equation}.

Firstly, notice that if $i\in[\newell]$ is such that $i\notin \Image f_2$ then 
$\FunctionG(\lambda_1,\dots,\lambda_\newell)$ does not depend on the variable 
$\lambda_i$ thus $\Delta_{\lambda_i} \FunctionG(\lambda_1,\dots,\lambda_\newell)=0$
and thus the contribution of $\FunctionG$ to \eqref{eq:here-contribution} vanishes. 
Thus it is enough to consider only surjective functions
$f_2:[n]\rightarrow[\newell]$.

Secondly, $\FunctionG(\lambda_1,\dots,\lambda_\newell)$ is a Laurent polynomial of degree at most
$n+1-2|C(\sigma_2)|$, thus in order for the coefficient of $A^{n+1-2\newell}$ to be non-zero, we must have
$|C(\sigma_2)|\leq\newell$.

The above two observations imply that in order to have a nontrivial contribution we must have $|C(\sigma_2)|=\newell$
and $f_2:C(\sigma_2)\rightarrow[\newell]$ must be a bijection; we will assume this in the following.

\medskip

Assume that $f:[n]\rightarrow \N^2$ given by $f(i)=\big( f_1(i), f_2(i) \big)$ is not injective.
It follows that there exist $i\neq j$ with $i,j\in[n]$ such that $f(i)= f(j)$.
For a given $\sigma_1\in\Sym{n}$ we denote $\sigma'_1:=(i,j)\; \sigma_1$, where $(i,j)\in\Sym{n}$ denotes the transposition
interchanging $i$ and $j$.

Note that $f=f\circ (i,j)$ thus
\begin{multline*}
 \big( \text{$f_1$ is constant on each cycle of $\sigma_1$}\big) \iff \\ 
f_1=f_1 \circ \sigma_1 \iff f_1=f_1 \circ (i,j)\; \sigma_1 \iff \\ 
 \big( \text{$f_1$ is constant on each cycle of $\sigma_1'$}\big).
\end{multline*}

\smallskip

We will show now that
\begin{equation} 
\label{eq:equivalence-transitive}
\text{$\langle \sigma_1,\sigma_2\rangle$ is transitive} \iff 
  \text{$\langle \sigma'_1,\sigma_2\rangle$ is transitive}.
\end{equation}
We will show only that the left-hand side implies the right-hand side; the opposite implication will follow
by interchanging the values of $\sigma_1$ and $\sigma_1'$.

Consider the case when $i$ and $j$ belong to different cycles of $\sigma_1$. 
Then
$C(\sigma_1')=C(\sigma_1) \vee \big\{ \{i,j\} \big\}$ is the set-partition obtained from the set-partition $C(\sigma_1)$
by merging the two blocks containing $i$ and $j$.
The left-hand side of \eqref{eq:equivalence-transitive} 
implies that $C(\sigma_1)\vee C(\sigma_2)=1_n$ is the maximal partition, thus
$C(\sigma_1')\vee C(\sigma_2)= C(\sigma_1)\vee C(\sigma_2) \vee \big\{ \{i,j\} \big\}=1_n$ as well.
This implies the right-hand side of \eqref{eq:equivalence-transitive}.

Consider the case when $i$ and $j$ belong to the same cycle of $\sigma_1$. Then
$C(\sigma_1)=C(\sigma'_1) \vee \big\{ \{i,j\} \big\}$.
Since $f_2:C(\sigma_2)\rightarrow[\newell]$ is a bijection, the equality $f_2(i)=f_2(j)$ implies that
$i$ and $j$ belong to the same cycle of $\sigma_2$. It follows that
\begin{multline*} 
1_n= C(\sigma_1) \vee C(\sigma_2) = \big( C(\sigma_1') \vee \{\{i,j\}\} \big) \vee C(\sigma_2) = \\
 C(\sigma_1') \vee \big( \{\{i,j\}\}  \vee C(\sigma_2) \big) = 
 C(\sigma_1') \vee C(\sigma_2). 
\end{multline*}
The latter is equivalent to the right-hand side of \eqref{eq:equivalence-transitive}.
 
\smallskip

The equivalence \eqref{eq:equivalence-transitive} implies that
$\sigma_1$ contributes to the sum within $\FunctionG$ in \eqref{eq:curly-bracket} if and only if 
$\sigma_1'$ contributes to this sum.
The map $\sigma_1\mapsto\sigma_1'$ is an involution without fixpoints.
It is easy to check that the contributions of $\sigma_1$ and $\sigma_1'$ to $[A^{n+1-2\newell}]\FunctionG$ cancel
each other.

In this way we proved that the contribution of a summand in \eqref{eq:curly-bracket}
to \eqref{eq:here-contribution} vanishes unless $f$ is injective.
In other words, if we introduce a new function $\widetilde{\Chtt_n}$
given by \eqref{eq:curly-bracket} in which we \emph{additionally restrict the summation}
only to functions $f$ which are injective, this would not change the value of
\eqref{eq:here-contribution} and  
\begin{multline}
\label{eq:almost-home}  
[A^{n+r-2k}] \Delta_{\lambda_1} \cdots \Delta_{\lambda_k} F_2^{\sym}(\lambda_1,\dots,\lambda_k)    =\\ 
[A^{n+r-2k}] \Delta_{\lambda_1} \cdots \Delta_{\lambda_k} \Chtt_n(\lambda_1,\dots,\lambda_k)    =\\
    [A^{n+r-2k}] \Delta_{\lambda_1} \cdots \Delta_{\lambda_k} \widetilde{\Chtt_n}(\lambda_1,\dots,\lambda_k).
\end{multline}

The injectivity assertion on $f$ implies that 
the function $\widetilde{\Chtt_n}$ fulfills an analogue of \eqref{eq:small-boxes-vanishes}
and henceforth an analogue of \eqref{eq:small-boxes-vanishes-B} as well.
It follows that the right-hand side of \eqref{eq:almost-home} vanishes if $|\lambda|+k<n$.

The latter observation together with \eqref{eq:small-boxes-vanishes-B} implies that
$F=F_1-F_2$ indeed fulfills condition  \ref{zero:vanishing}.

\bigskip

\emph{Condition \ref{zero:laurent}.}
\cref{prop:degree-laurent}
implies that $\Ch_n(\lambda)\in\Laurent$ is a Laurent polynomial of degree at most
$n-1$.

By \eqref{eq:normalized-embedding} it follows that the summand 
\[ \gamma^{n+1-|C(\sigma_1)|-|C(\sigma_2)|}
\ \Embed_{\sigma_1,\sigma_2} \in\Laurent\]
in \eqref{eq:top-top-top} is a Laurent polynomial of degree
\[ \big( n+1-|C(\sigma_1)|-|C(\sigma_2)| \big) + \big( |C(\sigma_1)| - |C(\sigma_2)| \big) \leq n-1 \]
and, henceforth, $\Chtt_n(\lambda)\in\Laurent$ is a Laurent polynomial of degree at most $n-1$.

By combining the above two observations it follows that $F$ indeed fulfills condition \ref{zero:laurent}.

\bigskip

\emph{Conclusion.}
We verified that $F$ fulfills the assumptions of \cref{lem:keylemma} for $r=1$;
it follows that $F= \Ch_n - \Chtt_n\in\Poly$ is of degree at most $n$.

\bigskip

\emph{The homogeneous part of degree $n$ is zero.}
We revisit \cref{sec:degree-bound-chttn};
there exist polynomials $H$ and $H'$ such that
\begin{align*}
\Chtt_n &= H(\gamma,\Sfunct_2,\Sfunct_3,\dots),  \\ 
\Ch_n   &= H'(\gamma,\Sfunct_2,\Sfunct_3,\dots).
\end{align*} 
The already proved bound on the degree of $F$ implies that --- with respect to the degrees given by \eqref{eq:gradacja-ZZtop} --- the difference
$H-H'$ is a polynomial of degree at most $n$.

The results from \cref{sec:degree-bound-chttn} imply that
--- with respect to the degrees given by \eqref{eq:gradacja-ZZtop} --- the homogeneous part of $H$ of degree $n$ is equal to zero.
The homogeneous part of $H'$ of degree $n$ is also
equal to zero by a parity argument (even vs.~odd) 
based on a result of Dołęga and F\'eray 
\cite[Proposition 3.7]{DoleegaFeray2014}.

In this way we proved that $H-H'$ is a polynomial of degree at most $n-1$ which completes the proof.
\end{proof}

\appendix

\section{Proof of \cref{coro:Kerov-Lassalle}}
\label{sec:proof-theo-kerov-lassalle}

\subsection{Expanders}
\label{sec:expander}

\begin{definition}
\label{def:weighted-expander}
We say that $(G,q)$ is an \emph{expander} if the following conditions are fulfilled:
\begin{enumerate}[label=(\alph*)]
\item $G$ is a bicolored graph with the set of black vertices $\V_\bullet$ and the set of white vertices
$\V_\circ$;

\item $q\colon \V_\bullet \to\{2,3,\dots\}$ is a function on the set of the black vertices;

\item $|\V_\circ|= \sum_{v \in \V_\bullet} \big( q(v)-1 \big)$,

 \item \label{enum:marriage} for every set $A\subset \V_\bullet$
such that $A\neq\emptyset$ and $A\neq \V_\bullet$ we require that 
\begin{multline*} \#\big\{ v\in \V_\circ: \text{$v$ is connected to at least one vertex in $A$}\big\} > \\
\sum_{i\in A} \big( q(i)-1 \big).   
\end{multline*}
\end{enumerate}
\end{definition}

\subsection{Kerov--Lassalle polynomials and expanders}

\newcommand{\G}{\mathcal{G}}

\begin{proposition}
\label{lem:extract-kerov}
Let $F\in\Poly$, let $\mathcal{G}$ be a finite collection of connected bicolored graphs and
let $\mathcal{G}\ni G\mapsto {m}_G\in \Q[\gamma]$ 
be a function on it. Assume that for each $\lambda\in\AllYoung$
$$ F(\lambda) = \sum_{G\in\mathcal{G}}  
 {m}_G\ \Embed_G(\lambda).
$$

Then the Kerov--Lassalle polynomial for $F$ is explicitly given by
\[ F  =  \sum_{G \in \G} \sum_{q} (-{m}_G) \prod_{v\in \V_\bullet(G)} \Rfunct_{q(v)} ,\]
where the sums run over $G$ and $q$ for which $(G,q)$ is an expander.
\end{proposition}
\begin{proof}
This kind of result was proved in the special case $A=1$, $\gamma=0$ in 
our joint work with Dołęga and F\'eray \cite{DolegaFeraySniady2008}.
In the following we will explain how to extend that result to our more general setup.

\medskip

Our goal is to find a multivariate polynomial $K$ (with coefficients in $\Q[\gamma]$) 
with the property that
\[ F= K( \Rfunct_2,\Rfunct_3,\dots ).\]
We shall reuse the ideas presented in the proof of 
\cref{prop:generate-the-same=free}.
Our current goal can be reformulated as expressing 
\emph{the anisotropic Stanley polynomial for $F$} as the polynomial $K$ in terms of 
\emph{the anisotropic Stanley polynomial for $\Rfunct_2$},
\emph{the anisotropic Stanley polynomial for $\Rfunct_3$},\dots
with the coefficients in $\Q[\gamma]$.

\cref{lem:stanley-polunomial-isotropic-and-anisotropic-the-same} shows
equalities between the Stanley polynomials in the isotropic setup
and in its anisotropic counterpart, thus the original problem 
is equivalent to the following one: we define
\begin{equation}
\label{eq:isotropic-Stanley-formula}
\bar{F}(\lambda) := \sum_{G\in\mathcal{G}}  (-1)^{|\V_\bullet(G)|}\ {m}_G\  N_G(\lambda)
\end{equation}
and we ask \emph{how to express the function $\bar{F}$ in terms of the isotropic free cumulants:}
\[ \bar{F}= K( R_2,R_3,\dots )?\]
This problem has been explicitly solved in \cite{DolegaFeraySniady2008}
for the special case when
$\bar{F}=\Ch_n^{A=1}$ is the character of the symmetric groups and \eqref{eq:isotropic-Stanley-formula}
takes a specific form of the Stanley's character formula.
However, as we explained in a joint work with F\'eray \cite[Lemma 4.2]{FeraySniady2011}, 
the argument holds for any polynomial function $\bar{F}$ 
(note that the sign in \cite[Lemma 4.2]{FeraySniady2011} is incorrect).
\end{proof}

\subsection{Proof of \cref{coro:Kerov-Lassalle}}
\label{sec:proof-of-theo:kerov-lassalle}

\cref{coro:Kerov-Lassalle} is a consequence of the following more precise result.

\begin{restatable}%
{theorem}{kerovlassallefortop} 
\label{theo:kerov-lassalle}
For each $n\geq 1$
the homogeneous part of degree $n+1$ 
of Kerov--Lassalle polynomial for $\Ch_n$ is given by
\begin{equation}
\label{eq:kerov-lassalle-exact}
\Chtt_n = 
\sum_{M}  
\gamma^{n+1-|\V(M)|}
\sum_{\substack{q\colon \V_\bullet(G) \to\{2,3,\dots\} \\ \text{$(M,q)$ is an expander}}} 
\prod_{v\in \V_\bullet(G)} \Rfunct_{q(v)}
,
\end{equation}
where the sum runs over \emph{rooted, oriented, bicolored, connected maps $M$ with $n$ unlabeled edges}.
\end{restatable}
\begin{proof}
It is a direct consequence of \cref{coro:nonoriented-maps} and \cref{lem:extract-kerov}.
\end{proof}

\section*{Acknowledgments}

I thank Maciej Dołęga and Valentin F\'eray for several years of collaboration on topics related to the current paper.
A part of \cref{sec:jack-polynomials-motivations} was written by Maciej Dołęga.

Research supported by \emph{Narodowe Centrum Nauki}, grant number \linebreak 2014/15/B/ST1/00064.

\bibliographystyle{alpha}
\bibliography{biblio}

\end{document}

%% file: preamble.tex
\usepackage{mathrsfs}

\usepackage{hyperref}

\usepackage{times}
\usepackage{amssymb}

\usepackage{enumitem}

\usepackage{amsthm}
\usepackage{thmtools,thm-restate}
\usepackage[capitalise,noabbrev]{cleveref}

\declaretheorem[numberwithin=section]{lemma}
\declaretheorem[sibling=lemma]{theorem}
\declaretheorem[sibling=lemma]{conjecture}
\declaretheorem[sibling=lemma]{corollary}
\declaretheorem[sibling=lemma]{proposition}

\declaretheorem[sibling=lemma,style=remark]{example}

\declaretheorem[sibling=lemma,style=remark]{definition}

\Crefname{inconvenientdefinition}{inconvenientdefinition}{{inconvenientdefinition}}

\usepackage{color}
\usepackage{gensymb}
\usepackage[usenames,dvipsnames,svgnames,table]{xcolor}

\usepackage{tikz}
\usepackage[backgroundcolor=yellow]{todonotes}
\usetikzlibrary{matrix,arrows,calc}
\usetikzlibrary{decorations.markings}
\usetikzlibrary{patterns}
\usetikzlibrary{snakes}

\usepackage{subfig}
\usepackage{xspace}
\usepackage[T1]{fontenc}
\usepackage[utf8]{inputenc}
\usepackage{bm}

\newcommand{\bigexists}{%
\mathop{\lower0.75ex\hbox{%
   \scalebox{1.7}{\ensuremath{\exists}}}}\limits}

\DeclareMathOperator{\ddeg}{deg}
\DeclareMathOperator{\content}{content}
\newcommand{\acontent}{\text{$\alpha$-$\content$}}

\DeclareMathOperator{\Tr}{Tr}
\DeclareMathOperator{\tr}{tr}

\DeclareMathOperator{\degg}{deg}
\DeclareMathOperator{\id}{id}
\DeclareMathOperator{\Stab}{Stab}
\DeclareMathOperator{\sym}{sym}
\DeclareMathOperator{\St}{St}

\newcommand{\Sym}[1]{\mathfrak{S}(#1)}

\newcommand{\Embed}{\mathfrak{N}}
\newcommand{\Sfunct}{\mathcal{S}}
\newcommand{\Tfunct}{\mathcal{T}}
\newcommand{\Rfunct}{\mathcal{R}}

\newcommand{\disjoint}{\bullet}
\newcommand{\row}{\otimes}

\newcommand{\V}{\mathcal{V}}

\newcommand{\E}{\mathcal{E}}
\newcommand{\EE}{\mathbb{E}}

\newcommand{\Z}{\mathbb{Z}}
\newcommand{\Q}{\mathbb{Q}}

\newcommand{\N}{\mathbb{N}}
\newcommand{\R}{\mathbb{R}}
\newcommand{\Alg}{\mathcal{A}}
\newcommand{\AlgB}{\mathcal{B}}
\newcommand{\AlgC}{\mathcal{C}}
\newcommand{\condExp}[2]{\EE_{#1}^{#2}}
\newcommand{\condKumu}[2]{k_{#1}^{#2}}

\newcommand{\AllYoung}{\mathbb{Y}}

\newcommand{\Poly}{\mathscr{P}}

\newcommand{\functions}{\mathscr{R}}

\DeclareMathOperator{\Image}{Im}
\DeclareMathOperator{\Ch}{Ch}
\newcommand{\Chtt}{\Ch^{\ttwisted}}
\DeclareMathOperator{\ttwisted}{top}
\DeclareMathOperator{\ttop}{top}

\newcommand{\poly}{p}

\newcommand{\MagicSet}{\mathcal{X}_n}

\newcommand{\Laurent}{\Q\left[A,A^{-1}\right]}

\numberwithin{equation}{section}

\newcommand{\faceAfill}{red!10}
\newcommand{\faceBfill}{blue!20}

\DeclareMathOperator{\partitions}{Part}

\newcommand{\kumuDisjointPoint}{\kappa_{\disjoint}}

\newcommand{\kumuPointRow}{\kappa^{\row}}
\newcommand{\kumuDisjointRow}{\kappa_{\disjoint}^{\row}}

%% file: git-commit-id.txt
a7a3d19

%% file: top-degree-Jack-characters.bbl
\begin{thebibliography}{DF{\'S}14}

\bibitem[Bia98]{Biane1998}
Philippe Biane.
\newblock Representations of symmetric groups and free probability.
\newblock {\em Adv. Math.}, 138(1):126--181, 1998.

\bibitem[BO05]{BorodinOlshanski2005}
Alexei Borodin and Grigori Olshanski.
\newblock Z-measures on partitions and their scaling limits.
\newblock {\em European Journal of Combinatorics}, 26(6):795--834, 2005.

\bibitem[Bri69]{Brillinger}
David~R. Brillinger.
\newblock The calculation of cumulants via conditioning.
\newblock {\em Annals of the Institute of Statistical Mathematics},
  21(1):215--218, 1969.

\bibitem[CJ{\'S}17]{Czyzewska-Jankowska2017}
Agnieszka Czy\.zewska-Jankowska and Piotr {\'S}niady.
\newblock Bijection between oriented maps and weighted non-oriented maps.
\newblock {\em Electron. J. Combin.}, 24(3):Paper 3.7, 34, 2017.

\bibitem[DF16]{DoleegaFeray2014}
Maciej Do{\l}\k{e}ga and Valentin F{\'e}ray.
\newblock Gaussian fluctuations of {Y}oung diagrams and structure constants of
  {J}ack characters.
\newblock {\em Duke Math. J.}, 165(7):1193--1282, 2016.

\bibitem[DF{\'S}10]{DolegaFeraySniady2008}
Maciej Do{\l}\k{e}ga, Valentin F{\'e}ray, and Piotr {\'S}niady.
\newblock Explicit combinatorial interpretation of {K}erov character
  polynomials as numbers of permutation factorizations.
\newblock {\em Adv. Math.}, 225(1):81--120, 2010.

\bibitem[DF{\'S}14]{DolegaFeraySniady2013}
Maciej Do{\l}\k{e}ga, Valentin F{\'e}ray, and Piotr {\'S}niady.
\newblock Jack polynomials and orientability generating series of maps.
\newblock {\em S\'em. Lothar. Combin.}, 70:Art. B70j, 50, 2014.

\bibitem[DH92]{DiaconisHanlon1992}
Persi Diaconis and Phil Hanlon.
\newblock Eigen-analysis for some examples of the metropolis algorithm.
\newblock {\em Contemporary Mathematics}, 138:99--117, 1992.

\bibitem[Do{\l}17]{Dolega2017a}
M.~Do{\l}{\k{e}}ga.
\newblock {Top degree part in $b$-conjecture for unicellular bipartite maps}.
\newblock {\em Electron. J. Combin}, 24(3):Paper 3.24, 39, 2017.

\bibitem[D{\'S}18]{DolegaSniady2014}
Maciej Do{\l}\k{e}ga and Piotr {\'S}niady.
\newblock Gaussian fluctuations of {J}ack-deformed random {Y}oung diagrams.
\newblock To appear in Probability Theory and Related Fields. Preprint
  \texttt{arXiv:1704.02352}, 2018.

\bibitem[F{\'S}11a]{FeraySniady2011a}
Valentin F{\'e}ray and Piotr {\'S}niady.
\newblock Asymptotics of characters of symmetric groups related to {S}tanley
  character formula.
\newblock {\em Ann. of Math. (2)}, 173(2):887--906, 2011.

\bibitem[F{\'S}11b]{FeraySniady2011}
Valentin F{\'e}ray and Piotr {\'S}niady.
\newblock Zonal polynomials via {S}tanley's coordinates and free cumulants.
\newblock {\em J. Algebra}, 334:338--373, 2011.

\bibitem[GJ96]{Goulden1996}
I.~P. Goulden and D.~M. Jackson.
\newblock Connection coefficients, matchings, maps and combinatorial
  conjectures for {J}ack symmetric functions.
\newblock {\em Trans. Amer. Math. Soc.}, 348(3):873--892, 1996.

\bibitem[IO02]{IvanovOlshanski2002}
Vladimir Ivanov and Grigori Olshanski.
\newblock Kerov's central limit theorem for the {P}lancherel measure on {Y}oung
  diagrams.
\newblock In {\em Symmetric functions 2001: surveys of developments and
  perspectives}, volume~74 of {\em NATO Sci. Ser. II Math. Phys. Chem.}, pages
  93--151. Kluwer Acad. Publ., Dordrecht, 2002.

\bibitem[Jac71]{Jack1970/1971}
Henry Jack.
\newblock A class of symmetric polynomials with a parameter.
\newblock {\em Proc. Roy. Soc. Edinburgh Sect. A}, 69:1--18, 1970/1971.

\bibitem[Kan93]{Kaneko1993}
Jyoichi Kaneko.
\newblock Selberg integrals and hypergeometric functions associated with jack
  polynomials.
\newblock {\em SIAM journal on mathematical analysis}, 24(4):1086--1110, 1993.

\bibitem[Ker93]{Kerov1993gaussian}
Serguei Kerov.
\newblock Gaussian limit for the {P}lancherel measure of the symmetric group.
\newblock {\em C. R. Acad. Sci. Paris S{\'e}r. I Math.}, 316(4):303--308, 1993.

\bibitem[Ker00]{Kerov2000}
S.~V. Kerov.
\newblock Anisotropic young diagrams and jack symmetric functions.
\newblock {\em Funct. Anal. Appl.}, 34:41--51, 2000.

\bibitem[KO94]{KerovOlshanski1994}
Serguei Kerov and Grigori Olshanski.
\newblock Polynomial functions on the set of {Y}oung diagrams.
\newblock {\em C. R. Acad. Sci. Paris S{\'e}r. I Math.}, 319(2):121--126, 1994.

\bibitem[Las08]{Lassalle2008a}
Michel Lassalle.
\newblock A positivity conjecture for {J}ack polynomials.
\newblock {\em Math. Res. Lett.}, 15(4):661--681, 2008.

\bibitem[Las09]{Lassalle2009}
Michel Lassalle.
\newblock Jack polynomials and free cumulants.
\newblock {\em Adv. Math.}, 222(6):2227--2269, 2009.

\bibitem[LZ04]{LandoZvonkin2004}
Sergei~K. Lando and Alexander~K. Zvonkin.
\newblock {\em Graphs on surfaces and their applications}, volume 141 of {\em
  Encyclopaedia of Mathematical Sciences}.
\newblock Springer-Verlag, Berlin, 2004.
\newblock With an appendix by Don B. Zagier, Low-Dimensional Topology, II.

\bibitem[Mat08]{Matsumoto2008}
Sho Matsumoto.
\newblock Jack deformations of {P}lancherel measures and traceless {G}aussian
  random matrices.
\newblock {\em Electron. J. Combin.}, 15(1):Research Paper 149, 18, 2008.

\bibitem[Nak96]{Nakajima1996}
Hiraku Nakajima.
\newblock {J}ack polynomials and {H}ilbert schemes of points on surfaces.
\newblock arXiv preprint alg-geom/9610021, 1996.

\bibitem[OO97]{OkounkovOlshanski1997}
Andrei Okounkov and Grigori Olshanski.
\newblock Shifted {J}ack polynomials, binomial formula, and applications.
\newblock {\em Math. Res. Lett.}, 4(1):69--78, 1997.

\bibitem[{\'S}ni06]{Sniady2006c}
Piotr {\'S}niady.
\newblock Gaussian fluctuations of characters of symmetric groups and of
  {Y}oung diagrams.
\newblock {\em Probab. Theory Related Fields}, 136(2):263--297, 2006.

\bibitem[{\'S}ni16]{Sniady2016a}
Piotr {\'S}niady.
\newblock Top degree of {J}ack characters and enumeration of maps.
\newblock arXiv:1506.06361v2, 2016.

\bibitem[VK81]{VershikKerov1981a}
A.~M. Vershik and S.~V. Kerov.
\newblock Asymptotic theory of the characters of a symmetric group.
\newblock {\em Funktsional. Anal. i Prilozhen.}, 15(4):15--27, 96, 1981.

\end{thebibliography}
